\documentclass[a4paper,11pt,leqno]{article}   
\usepackage{amssymb, amsmath, amsthm, graphicx, color, epsfig, 
amsfonts, tikz, url}

\usepackage{enumitem}

\usepackage{thmtools} 
\newtheorem{thm}{Theorem}[section]
\newtheorem{lem}[thm]{Lemma}

\newtheorem{rem}{Remark}[section]

\numberwithin{equation}{section}

\renewcommand{\a}{\alpha}
\renewcommand{\b}{\beta}
\newcommand{\e}{\varepsilon}

\newcommand{\fa}{\varphi}

\newcommand{\la}{\lambda}

\newcommand{\Ga}{\Gamma}

\newcommand{\Om}{\Omega}

\def\ap{{\a_+}}
\def\am{{\a_-}}
\def\apm{{\a_\pm}}

\def\us{{U_0}}
\def\usz{{U_{0z}}}
\def\uszz{{U_{0zz}}}
\def\use{U_{0e_j}}

\def\ul{U_1}
\def\ulz{U_{1z}}
\def\ulzz{U_{1zz}}

\def\pe{\partial_e}

\def\pvi{\partial_{x_i}}
\def\pei{\partial_{e_i}}

\def\pvj{\partial_{x_j}}
\def\pej{\partial_{e_j}}

\def\D{D_{ij}}
\def\Dk{D_{ik}}

\def\o{\mathcal{O}}

\def\vp{{\varphi}}

\def\mij{\mu_{ij}}
\def\msij{\mu_{ij}^1}
\def\mlij{\mu_{ij}^2}

\def\ml{{\mathcal{L}}}

\def\R{{\mathbb{R}}}

\allowdisplaybreaks

\title{Motion of sharp interface of Allen-Cahn equation with anisotropic nonlinear diffusion}
\author{Tadahisa Funaki$^\ast$ and Hyunjoon Park$^\dagger$}
\date{\today}

\begin{document}
\maketitle

\begin{abstract}

We consider the Allen-Cahn equation with nonlinear anisotropic diffusion 
and derive anisotropic direction-dependent curvature flow under the 
sharp interface limit.  The anisotropic curvature flow was already 
studied, but its derivation is new.  We prove both generation and 
propagation of the interface.  For the proof we construct sub- 
and super-solutions applying the comparison theorem. The problem 
discussed in this article naturally appeared in the study of the 
interacting particle systems, especially of non-gradient type. 
The Allen-Cahn equation obtained from systems of gradient type has 
a simpler nonlinearity in diffusion and leads to isotropic mean-curvature flow. 
We extend those results to anisotropic situations.

\end{abstract}

\footnote{\hskip -6.5mm
$^\ast$Beijing Institute of Mathematical Sciences and Applications, 
No.\ 544, Hefangkou Village, Huaibei Town, Huairou District, Beijing 101408, China.
e-mail: funaki@ms.u-tokyo.ac.jp \\
$^\dagger$Meiji Institute for Advanced Study of Mathematical Sciences, Meiji University, 4-21-1 Nakano, Tokyo 164-8525, Japan. 
e-mail: hyunjoonps@gmail.com

\vskip 2mm
\noindent

\noindent
}

\section{Introduction}

The Allen-Cahn equation with nonlinear diffusion has a natural physical
background.  However, compared to those with linear diffusion,
they seem to be less studied.  In general, nonlinear partial differential
equations, which describe macroscopic phenomena, are derived from
microscopic systems via a certain scaling limit especially under an
averaging effect due to local ergodicity of the system; 
see \cite{Sp91}, \cite{KL}.  In particular,
the linear Laplacian arises at macroscopic level when molecules at
microscopic level evolve independently. However, if molecules evolve
with interaction, we obtain a nonlinear Laplacian instead of linear.
Especially when the interaction of microscopic particles has    
a special structure called of gradient type, which gives a good
cancellation in scaling limit, we obtain nonlinearity of the form
$\Delta \varphi(u)$ for particle density $u$ with a certain nonlinear
increasing function $\varphi$; see Remark \ref{rem:1.1}-case (i), 
\cite{F18} (writing $P$ instead of $\fa$).  More generally,
from microscopic systems called of non-gradient type, we obtain
a nonlinear differential operator of second order
$\sum_{i,j=1}^N \partial_{x_i}\{D_{ij}(u) \partial_{x_j}u\}$
of divergence form.  The diffusion coefficients $\{D_{ij}(u)\}$
are known to be described by the so-called Green-Kubo formula;
see \cite{Sp91}, \cite{F23}, \cite{FUY}.  The reaction term $f(u)$ in the Allen-Cahn
equation reflects the creation and annihilation of microscopic
particles.

In this article we study the Allen-Cahn equation with nonlinear anisotropic diffusion. Namely, we consider the following Cauchy problem of partial differential equation:
\begin{align}\tag{$P^\e$}\label{eq:pe}
    \begin{cases}
        \ml(u^\e)
        :=
        \partial_t u^\e 
        - \sum\limits_{i,j=1}^N \partial_{x_i}\big\{ D_{ij}(u^\e) \partial_{x_j} u^\e \big\} 
        -
        \dfrac{1}{\e^2}f(u^\e)
        &\text{in}~ \Omega \times (0,T),
        \\
        \dfrac{\partial u^\e}{\partial \nu} = 0
        &\text{on}~ \partial \Omega \times (0,T),
        \\
        u^\e(x,0) = u_0(x)
        &x \in \Omega,
    \end{cases}
\end{align}
where $N \ge 2$ is the spatial dimension, $\Om$ is a smooth bounded domain in $\R^N$, $\e$ is a small positive number, $\nu$ is the outward normal vector on the boundary $\partial \Om$ and $u_0(x)$ is a bounded and $C^2$ function in $\Om$. The function $f$ is a bistable reaction term with three roots $f(\alpha_+) = f(\a) = f(\a_-) = 0, \a_- < \a < \a_+$ and satisfying 
\begin{align}\label{cond:f-bistable}
    f'(\a_\pm) < 0,~ \nu := f'(\a) > 0,~f \in C^2(\R).
\end{align}
The term $(\D(s))_{1 \le i,j \le N}$ is symmetric and strictly positive definite matrix for $s \in \R$. We further assume the existence of some positive constants $c_D, C_D > 0$ satisfying 
\begin{align}\label{cond:D_pos}
    c_D \le \sum_{i,j = 1}^N \D(s) \eta_i \eta_j \le C_D
    ,~
    \Vert \D \Vert_{C^3(\R)} \le C_D
\end{align}
where $s \in \R, \eta \in \R^d, |\eta| = 1$. Moreover, we assume equipotential condition to $\D$ and $f$:
\begin{align}\label{cond:Df_equi}
    \int_{\a_-}^{\a_+} \D(s) f(s) ds = 0,
\end{align}
for all $1 \le i,j \le N$.

For the initial condition $u_0(x)$ we assume that $u_0 \in C^2(\Om)$. Throughout the paper, we define $c_0$ as follows:
\begin{align}
	c_0
	&:= || u_0 || _{C^2 \left( \Om \right)}.
    \label{cond:C0}
\end{align}
Furthermore, we define $\Gamma_0$ by
\begin{align*}
	\Gamma_0 
	:= 
	\{
		x \in \Om: u_0(x) = {\color{black}\a}
	\}
\end{align*}
and we suppose $\Gamma_0$  is a $C^{4+\nu}, 0 < \nu < 1$, hypersurface without boundary such that
\begin{align}
	\Gamma_0 \Subset \Om, \nabla u_0(x) \cdot n(x) \neq 0 \text{ if}~ x \in \Gamma_0, \label{cond:gamma0_normal} \\
	u_0 > \alpha \text{ in } \Om_0^+, ~~~~~~ u_0 < \alpha \text{ in } \Om_0^- ,\label{cond:u0_inout}
\end{align}
where $\Om_0^-$ denotes the region enclosed by $\Gamma_0$, $\Om_0^+$ the region enclosed between $\partial \Om$ and $\Gamma_0$ and $n$ is the outward normal vector to $\Ga_0$.

As $\e \to 0$, the reaction term prevails over the diffusion term, thus the limit solution will take either $\a_+$ or $\a_-$ and a hypersurface $\Ga_t$ (which we call an interface) occurs that separates the two stable steady states. From this observation we expect two stages to take place for the solution $u^\e$ of \eqref{eq:pe}: (I) in the early stage the diffusion term is negligible compared to the reaction term $\e^{-2} f(u)$, hence the solution $u^\e$ can be approximated by ordinary differential equation $u_t = \e^{-2}f(u)$. This implies that the solution $u^\e$ quickly converges close to either $\ap$ or $\am$, creating a steep transition layer. (II) After the creation of the steep transition layer, it starts to propagate. And, from the limiting behavior of $u^\e$, one can expect that the movement of this steep transition layer can be described by an interface $\Ga_t$. In fact, the interface propagates according to the following motion equation (see Section \ref{sec:formal_asympt} for details) 
\begin{align}\tag{$P^0$}\label{P0}
    \begin{cases}
        V_n 
        =
        - \sum_{i,j = 1}^N \mij(n) \pvi n_j 
        &
        \text{on~}\Gamma_t,
        \\
        \Gamma_t\Big|_{t = 0} = \Gamma_0,
    \end{cases}
\end{align}
where $V_n$ is the outward normal velocity of $\Ga_t$, $n = (n_i)_{i = 1, \cdots N}$ is the outward normal vector to $\Ga_t$, and $\mij$ is a function on $\mathbb{S}^{N - 1} = \{e\in \R^N; |e|=1\}$ defined as
\begin{align*}
    \mij(e)
    &=
    \dfrac{1}{\la(e)} 
    \int_\am^\ap 
    \left[
    \D(s) \sqrt{W_e(s)} 
    -\dfrac{\partial_{e_i} (W_e(s))}{2} 
    \pej
    \left(
        \dfrac{ a_e(s)}{\sqrt{W_e(s)}}
    \right)
    \right]
    ds,
    \\
    \la(e)
    &=
    \int_\am^\ap \sqrt{W_e(s)} ds
    ,~
    W_e(s) = -2 \int_\am^s a_e(s) f(s) ds
    ,~
    a_e(s) = e \cdot D(s) e.
\end{align*}
In our setting, the matrix $(\mij)$ becomes dependent on $n$ which gives an anisotropic feature to the interface motion. The well-posedness of the problem \eqref{P0} will be shown in Section \ref{sec:formal_asympt}. Hereafter, we let $T > 0$ be the time that $\Ga_t$ exists on $[0,T]$ and denote $\Om_t^-$ be the region enclosed by $\Ga_t$ and $\Om_t^+$ be the region enclosed by $\Ga_t$ and $\partial \Om$.

\begin{rem}  \label{rem:1.1}
The motion equation \eqref{P0} is general in the sense that it matches with the motion equation of the Allen-Cahn equation with (i) isotropic nonlinear diffusion or (ii) anisotropic diffusion without the nonlinear diffusivity. 
In the case of (i), where $D_{ii}(s) = \vp'(s)$ with a smooth increasing function $\vp$ and $\D = 0$ if $i \neq j$, we no longer have $e$ dependency in every terms, thus the second term in $\mij(e)$ vanishes and the motion equation becomes
\begin{align*}
    V_n = -\tilde{\la} \kappa,
\end{align*}
where $\kappa$ is a mean curvature and $\tilde{\la}$ is some constant which depends on $\vp$ and $ f$, and this coincides with the result of \cite{EFHPS}, \cite{EFHPS-2}, \cite{FvMST}. 
In the case of (ii), the matrix $\D$ no longer depends on $s$ and the term $a_e$ depends only on $e$, thus these terms can be considered as constant during the computation of $\mij$, which yields 
\begin{align*}
    \mij(e)
    =
    \D 
    -{\pei (\sqrt{a_e}) \pej (\sqrt{a_e})}.
\end{align*}
This implies that $\mij$ becomes independent from the reaction term $f$ and the resulting motion equation can be simplified to
\begin{align*}
    \dfrac{V_n}{\sqrt{a_n}}
    =
    -
    {\rm div}
    \left(
        \dfrac{\partial_n a_n}{2\sqrt{a_n}}
    \right),
\end{align*}
which coincides with the motion equation introduced in \cite{AGHMS2009}, \cite{BP1996}, where $a_n = a_e$ with $e = n$. 
\end{rem}

The aim of this article is to rigorously prove that the solution $u^\e$ of \eqref{eq:pe} converges to a step function with boundary $\Ga_t$ following the anisotropic curvature flow \eqref{P0} as $\e$ tends to 0. For this we give an error estimate between $u^\e$ and the $\Ga_t$ by constructing a pair of sub- and super-solutions, thus implying that the solution $u^\e$ converges to the step function $\tilde{u}$, where
\begin{align*}
    \tilde{u}(x,t)
    =
    \begin{cases}
        \ap
        &
        \text{in} ~ \Om_t^+,
        \\
        \am
        &
        \text{in} ~ \Om_t^-,
    \end{cases}
    ~~
    \text{for}
    ~ t \in [0,T].
\end{align*}

We first give the result of the generation of the interface. This theroem implies that, given  artibrary initial condition satisfying \eqref{cond:C0}-\eqref{cond:u0_inout} the solution $u^\e$ creates a steep transition layer within a short time of $\o(\e^2|\ln\e|)$ around the initial interface $\Ga_0$ with width $\o(\e)$, separating the steady states $\ap$ and $\am$.
\begin{thm}\label{thm:gen}
Let $u^\e$ be the solution of the problem $(P^\e)$, $\eta_g$ be an arbitrary constant satisfying $0 < \eta_g < \eta_0$, where $\eta_0 := \min \{ \ap - \a, \a - \am \}$. Then, there exist positive constants $\e_0$ and $M_0$ such that, for all $\e \in (0, \e_0)$, the following holds: 

\begin{enumerate}[label = (\roman*)]
\item for all $x \in \Om$
\begin{align}\label{thm:gen_1}
	\am - \eta_g
	\leq
	u^\e(x,t^\e)
	\leq
	\ap + \eta_g,
\end{align}

\item if $u_0(x) \geq \a +  M_0 \e$, then
\begin{align}\label{thm:gen_2}
	u^\e(x,t^\e) \geq \ap - \eta_g,
\end{align}

\item if $u_0(x) \leq \a - M_0 \e$, then
\begin{align}\label{thm:gen_3}
	u^\e(x,t^\e) \leq \am + \eta_g.
\end{align}
\end{enumerate}
Here $t^\e = \nu^{-1} \e^2 \ln \e$, and recall \eqref{cond:f-bistable} for $\nu$.

\end{thm}

After the generation, the steep transition layer propagates due to the effect of the diffusion. The theorem below implies that this transition layer propagates close to $\Ga_t$ within the distance $\o(\e)$.

\begin{thm}\label{thm:prop}
Under the conditions given in Theorem \ref{thm:gen}, for any given $0 < \eta  < \eta_0$ there exist $\e_0 > 0, C_p > 0$ and $T > 0$ such that for any $\e \in (0,\e_0)$ and $t \in [t^\e, T]$ we have
\begin{align*}
u^\e (x,t) \in~
	\begin{cases}
    [\alpha_{-} - \eta, \alpha_{+} + \eta] 
	&\text{ if }
	x \in \Om,
    \\
	[\alpha_{+} - \eta, \alpha_{+} + \eta] 
	& \text{ if } x \in \Om_t^+ \setminus \mathcal{N}_{ \e C_p}(\Ga_t),
	\\
	[\alpha_{-} - \eta, \alpha_{-} + \eta] 
	& \text{ if } x \in \Om_t^- \setminus \mathcal{N}_{ \e C_p }(\Ga_t),
	\end{cases}
\end{align*}
where $\mathcal{N}_r(\Ga_t) := \{ x \in \Om, {\rm dist}(x, \Ga_t) \le r \}$ is the $r$-neighborhood of $\Ga_t$. 
\end{thm}

The rest of the paper is organized as follows.
In Section \ref{sec:formal_asympt} we give a formal asymptotic analysis to derive the motion equation \eqref{P0}. In case of anisotropic diffusion without $u$ dependency, formal derivation was done in \cite{BP1996} by using a Finsler geometry. In this article, it is difficult to use the similar approach due to the $u$ dependency, which leads us to take different method. The argument is based on the formal derivation of \cite{NMHS1999}, with the additional idea to describe the anisotropic effect. In addition, we will also show that \eqref{P0} possesses a unique smooth solution locally in time.

In Section \ref{sec:generation} we prove the generation of a steep transition layer within a short time of scale $\o(\e^2 |\ln \e|)$. For this we construct sub- and super-solutions using the solution of the ordinary differential equation $Y_\tau = f(Y)$. In Section \ref{sec:propagation} we construct another sub- and super-solutions by using two leading terms of the formal asymptotic expansion in Section \ref{sec:formal_asympt}. 

Let us mention some earlier works on anisotropic problems related to \eqref{eq:pe}. In \cite{BHW2004} Benes, Hilhorst and Weidenfeld study the case of anisotropic diffusion without the $u$ dependency, showing the generation and the propagation by anisotropic  curvature flow. Later, in \cite{AGHMS2009} Alfaro et. al. improve the previous work by considering the heterogeneity and study more general interface motion. The anisotropic diffusivity in these papers is more general in the sense that it covers the ellipsoidal diffusion. We also mention the work of Garcke, Nestler and Stoth \cite{GNS1998} for a related work of generalized anisotropic diffusion having the $u$ dependency in the context of multi-phase system.

The problem of the singular limit for \eqref{eq:pe} discussed in this article
naturally appeared in the study of the interacting particle system called
Glauber-Kawasaki dynamics, especially, of non-gradient type; see \cite{F23}.
The models of gradient type led to the same problem but only for
isotropic nonlinear diffusion of type (i) discussed in Remark 1.1;
see \cite{EFHPS}, \cite{FvMST}.  Our results hold also on the $N$-dimensional torus
${\mathbb{T}}^N \cong [0,1)^N$ with the periodic boundary condition
and therefore, they are applicable in the setting of \cite{F23}.

Another comment is that our equation does not have any proper energy
functional, that is, it cannot be expressed as a gradient flow and
therefore, the method of the $\Gamma$-convergence seems not working.

\section{Formal asymptotic expansion}\label{sec:formal_asympt}

In this section we give a formal asymptotic expansion to derive the interface motion corresponding to Problem \eqref{eq:pe} using the argument introduced in \cite{NMHS1999}. Even though the computation in this section is formal, it provides a helpful intuition for the analysis in later sections.

We start from the assumption that $d^\e=d^\e(x,t)$ is the signed distance function to the interface $\Ga^\e_t := \{ x \in \Om, u^\e(x,t) = \a \}$ defined by
\begin{align*}
    d^\e(x,t)
    =
    \begin{cases}
        {\rm dist}(v,\Ga^\e_t)
        &
        \text{for}~x \in \Om^{\e,+}_t
        \\
        -{\rm dist}(v,\Ga^\e_t)
        &
        \text{for}~x \in \Om^{\e,-}_t,
    \end{cases}
\end{align*}
where $\Om^{\e,+}_t$ is the area enclosed by $\Ga^\e_t$ and $\Om^{\e,-}_t$ is the area enclosed between $\partial \Om^\e$ and $\Ga^\e_t$. Following the idea of \cite{NMHS1999} we assume that $d^\e$ has the expansion
\begin{align*}
    d^\e(x,t)
    =
    d_0(x,t) + \e d_1(x,t) + \e^2 d_2(x,t) + \cdots,
\end{align*}
and define $\Ga_t = \{ x \in \Om, d_0(x,t) = 0 \}$. In this way, $\Ga_t$ represents the interface of $u^\e$ as $\e \to 0$ and $d_0$ can be considered as the signed distance function of $\Ga_t$. 

We assume that $u^\e$ also has similar expansion to $d^\e$. Thus, away from the interface $\Ga_t$ we assume
\begin{align*}
    u^\e(x,t)
    =
    \alpha_\pm + \e u_1^\pm(x,t) + \e^2 u_2^\pm(x,t) + \cdots
    ~
    \text{in}
    ~Q_T^\pm
    ,~~
    Q_T^{ \pm} = \cup_{0 < t \le T} \big( \Om_t^{\pm} \times \{t\} \big).
\end{align*}
 Similarly, we assume $u^\e$ has the expansion
\begin{align*}
    u^\e(x,t)
    =
    \us(z,x,t; e) + \e \ul(z,x,t; e) + \e^2 U_2(z,x,t; e) + \cdots, z = \frac{d^\e}{\e}
\end{align*}
near the interface $\Ga_t$. The dependency of $e$ can be derived by considering the anisotropic diffusion on a level set of $u^\e$. For each level set of $u^\e$ the diffusion depends on two factors; the diffusivity matrix $D(u)$ and the direction of the diffusion $\nabla u^\e/|\nabla u^\e|$. The first factor can be described just by using $u^\e$ itself. And from the fact that the steep transition has of width $\o(\e)$, one can expect that the direction of the gradient $\nabla u^\e/|\nabla u^\e|$ can be approximated by $\nabla d^\e$. Moreover, since we expect $d^\e$ converge to $d_0$ as $\e \to 0$, we approximate $\nabla d^\e $ by $ \nabla d_0$. From now on, we denote $e = \nabla d_0$ in this section. 

To make the inner and outer expansions consistent, it is necessary that
\begin{align}\label{eqn_form1}
    \us(\pm \infty,x,t; e) = \alpha_\pm
    ,~~
    U_i(\pm\infty,x,t; e) = u_i^\pm(x,t),
\end{align}
for all $i \ge 1$.  
In this way, the function $\us$ represents the profile of the transition layer near the interface in a stretched variable.
Also, in order to normalize the function $U_i, i \ge 0$, using the fact that $\Ga^\e_t$ is the level set $u^\e$ of value $\a$ we assume
\begin{align}\label{eqn_form2}
    \us(0,x,t;e) = \a,~U_i(0,x,t;e) = 0,
\end{align}
where $i \ge 1$.

 Under this assumption, we search for the suitable candidate of $\us$ and $\ul$ by substituting the inner expansion into the equation \eqref{eq:pe}. By noting that  $|\nabla d^\e|=1$ near $\Ga^\e_t$, direct computation gives 
\begin{align*}
    \partial_t u &= \frac1\e \usz \partial_t d^\e +O(1), \\
    \frac1{\e^2} f(u) &= \frac1{\e^2} f(\us) + \frac1{\e} f'(\us) U_1 + \o(1)
\end{align*}
for the time derivative term and the reaction term. For the space derivative term, we first observe 
\begin{align*}
\partial_{x_j} u 
= \frac1\e \usz  \partial_{x_j} d^\e
+ \sum_{k = 1}^{N} U_{0 e_k} \partial_{x_j} \partial_{x_k} d^\e
+\ulz \partial_{x_j} d^\e
+ \pvj \us + \o(\e),
\end{align*}
where $\usz$ and $U_{0e_k}$ are the derivatives of $\us$ with respect to $z$ and $e_k$ respectively.   
Then,
\begin{align}
    \partial_{x_i}\{ D_{ij}(u) \partial_{x_j} u\}
    & =\frac1{\e^2} \Big(D_{ij}(\us) \usz\Big)_z 
    \left(\partial_{x_i} d_0 \, \partial_{x_j} d_0
    \right.\nonumber
    \\
    &\left.
        + \e ( \pvi d_0 \, \pvj d_1  + \pvi d_1 \, \pvj d_0)
    \right)\nonumber
    \\
    &+ \frac1\e
    \Big[ 
        \Big( D_{ij}(\us) \ul\Big)_{zz} \pvi d_0 \, \pvj d_0\nonumber
    \\
    &
    +
        D_{ij}(\us)\usz \, \partial_{x_i} \partial_{x_j} d_0\nonumber
    \\
    &
    +
    \sum_{k = 1}^N 2
    \left(
    	D_{ij}(\us)U_{0e_k}
    \right)_z
    \partial_{x_i} d_0 \, \partial_{x_j} \partial_{x_k} d_0\nonumber
    \\
    &
    +
    \pvi (\D(\us) \usz) \pvj d_0
    +
    (\D(\us) \pvj \us)_z \pvi d_0
    \Big]
    + \o(1).\label{eq:formalasymp}
\end{align}
We collect the terms of scale $\o(\frac{1}{\e^2})$ and $\o(\frac{1}{\e})$.
Taking the terms of order $\o(\frac1{\e^2})$, 
we have
\begin{align*}
\frac1{\e^2} 
\left[
\sum_{i,j=1}^N \big\{ D_{ij}(\us) 
\usz\big\}_z   \partial_{x_i} d_0 \, \partial_{x_j} d_0
 + f(\us)
 \right] = 0.
\end{align*}
As this equation holds independent of $x$ and $t$, we can assert $\us(z,x,t;e) = \us(z;e)$. Thus, considering the matching conditions \eqref{eqn_form1} and normalization condition \eqref{eqn_form2}, $\us(z;e)$ is the unique solution of
\begin{align}\label{eq:TW-e}
&\begin{cases}
    \big(a_e(\us)\usz\big)_z  + f(\us)
    =0, \quad z \in \R,
    \\
    \us(0) = \a,~ \us(\pm \infty) = \apm,
\end{cases}
    \\
    &
    ~~a_e(s) :=  e \cdot D(s) e
    ,~ 
    A_e(s) := \int_{\a_-}^s a_e(t) dt,\nonumber
\end{align}
where $\cdot$ denotes the inner product in $\R^d$.
The existence is guaranteed under the condition \eqref{cond:Df_equi}; see \cite{EFHPS-2}.

Next we consider the terms of order $\o(\frac1\e)$. Since $\us$ is a function independent of $x$, the last two terms in \eqref{eq:formalasymp} vanishes. This allows us to obtain the following equation for $\ul$
\begin{align}\label{eqn:TW_lin,formal}
\Big( a_e (\us)  \, U_1\Big)_{zz} + f'(\us) U_1 
= & \usz \partial_t d_0 
-
D_{ij} (\us) \, \usz \partial_{x_i} \partial_{x_j} d_0 \nonumber
\\
&- \Big(D_{ij}(\us) \usz\Big)_z ( \pvi d_0 \, \pvj d_1  + \pvi d_1 \, \pvj d_0) 
\nonumber
\\
&
- 
\left(
	\pei(a_e)(\us)~\use
\right)_z
 \pvi \pvj d_0,
\end{align}
where $\pei(a_e)$ is the derivative of $a_e$ with respect to $e_i$ and we omitted the sum $\sum_{i,j = 1}^N$. The left hand side can be seen as the linearized problem of \eqref{eq:TW-e}, thus the solvability condition is important in understanding the interface motion equation. We give here the lemma related to this, which comes from \cite{EFHPS-2}.
\begin{lem}\label{lem:TWlin_sol}
 Let $G(z)$ be a bounded function on $\R$ and $e \in \mathbb{S}^{N - 1}$. Then the problem
    \begin{align}\label{eq:TW-lin}
        \begin{cases}
            (a_e(\us)\psi)_{zz} + f'(\us) \psi 
            =
            G(z)
            ,~~z \in \R
            \\
            \psi(0) = 0
            ,~~\psi \in L^\infty(\R),
        \end{cases}
    \end{align}
    has a unique solution if and only if 
    \begin{align}\label{lem:TWlin_sol1}
        \int_\R G(z)  (A_e(\us(z)))_z  dz
        =
        0.
    \end{align}
    Moreover, the solution can be written as 
    \begin{align*}
        \psi(z)
        =
        \usz
        \int_0^z \dfrac{1}{(A_e(\us(\xi))_z)^2} 
        \left( \int_{-\infty}^\xi G(\zeta) A_e(\us(\zeta))_z d \zeta \right)
        d \xi.
    \end{align*}
\end{lem}

From this lemma, by considering the terms $\partial_t d_0, \, \pvi \pvj d_0, \, \pvi d_0  \pvi d_1 $ as coefficients, the solvability condition for $\ul$ in \eqref{eqn:TW_lin,formal} leads to the interface motion equation as follows 
\begin{align} \label{eq:d0-p}
    \partial_t d_0
    = \sum_{i,j = 1}^N \mij(\nabla d_0)  \pvi \pvj d_0
    ,
\end{align}
where 
\begin{align*}
    &\mij(e) = \msij(e) + \mlij(e), \\
    & \msij(e) = {\la(e)}^{-1}{\int_\R A_e(U_0)_z D_{ij}(U_{0}) \usz   dz}, \\
    & \mlij(e) = 2 {\la(e)}^{-1} {\sum_{k = 1}^N \int_\R A_e(U_0)_z \big( e_k \Dk(U_{0}) \use \big)_{z}  dz},\\
    & \la(e) = \int_\R A_e(U_0)_z \usz dz.
\end{align*}
Note that the term containing $\Big(D_{ij}(\us) \usz\Big)_z $ in \eqref{eqn:TW_lin,formal} does not appear in \eqref{eq:d0-p} since
\begin{align}
    \int_\R \Big(D_{ij}(\us) \usz\Big)_z (A_e(\us))_z dz
    &=
    - \int_\R \D(\us) \usz (A_e(\us))_{zz} dz\nonumber
    \\
    &=
    \int_\R \D(\us) f(\us) \usz dz\nonumber
    \\
    &= \int_\am^\ap \D(s) f(s) ds = 0,\label{eq:tw_1-1}
\end{align}
where the last inequality holds by \eqref{cond:Df_equi}. 
From \eqref{eq:d0-p} we now derive the interface motion equation \eqref{P0}. Since $\nabla d_0$ is equal to the outward normal vector to the interface $\Ga_t$ which we denote as $n$ and that $V = - \partial_t d_0$ we derive that
\begin{align*}
    V_n 
    =
    - \sum_{i,j = 1}^N \Big[\msij(n) + \mlij(n) \Big]\pvi n_j,
\end{align*}
thus with the initial condition $\Ga_0$ gives \eqref{P0}.

To understand the motion more clearly, we derive an explicit form of the coefficients in \eqref{eq:d0-p}. Note that by \eqref{eq:TW-e} we have
\begin{align}\label{eq:tw_1-2}
    A_e(\us)_z = \sqrt{W_e(\us)}
    ,~W_e(u) := - 2 \int_{\alpha_-}^u a_e(s) f(s) ds.
\end{align}
From this we can derive that
\begin{align}\label{eq:IM-1}
    \lambda(e)
    =
    \int_{\am}^\ap  \sqrt{W_e(s)} ds
    ,~~
    \la(e)\msij(e) = \int_\am^\ap \D(s) \sqrt{W_e(s)} ds.
\end{align}

For the term $\mlij$, we first need to understand the function $\use$. The existence of $\use$ is guaranteed by Lemma \ref{lem:TWlin_sol} and \eqref{eq:tw_1-1}; see Appendix of \cite{FH1988}. Moreover, by taking the derivative in $e_j$ directly to \eqref{eq:TW-e} we derive the following equation for $\use$
\begin{align*}
    (a_e(\us)\use)_{zz} + f'(\us) \use 
    =
    - (\pej (a_e)(\us) \usz)_z
    =
    -2 \sum_{k = 1}^N  e_k (D_{jk}(\us) \usz)_z.
\end{align*}
In addition, direct computation gives
\begin{align*}
    \int_\am^z A_e(\us)_z 
    \left( 
        e_k D_{jk}(\us)\usz 
    \right)_z~dz
    &=
    A_e(\us)_z e_k D_{jk}(\us)\usz 
    \\
    &+ \int_\am^\us e_k D_{jk}(s) f(s) ~ds
    \\
    &=
    A_e(\us)_z e_k D_{jk}(\us)\usz 
    - \dfrac{1}{4} \partial_{e_j} W_e(\us),
\end{align*}
where we omitted the summation $\sum_{k = 1}^N$. With this, we can obtain the explicit form of $\use$ by Lemma \ref{lem:TWlin_sol}
\begin{align}
    -\use
    &=
    \usz 
    \int_0^z 
        \dfrac{2 e_k D_{jk}(\us)\usz}{\sqrt{W_e(\us)}}
        - \dfrac{\partial_{e_j} W_e(\us)}{2W_e(\us)} 
    ~dz\nonumber
    \\
    &=
    \usz
    \int_0^z
        \dfrac{\partial_{e_j} (a_e) (\us)\usz}{\sqrt{W_e(\us)}}
        - \dfrac{\partial_{e_j} W_e(\us) a_e(\us) \usz}{2(W_e)^{3/2}(\us)} 
    ~dz\nonumber
    \\
    &=
    \usz
    \int_\a^\us
        \dfrac{\partial_{e_j} (a_e) (s)}{\sqrt{W_e(s)}}
        - \dfrac{\partial_{e_j} W_e(s) a_e(s)}{2(W_e)^{3/2}(s)} 
    ~ds\nonumber
    \\
    &=
    \usz
    \int_\a^\us
        \partial_{e_j}
        \left(
            \dfrac{ a_e(s)}{\sqrt{W_e(s)}}
        \right)
    ~ds.\label{eq:TW_lin_explicit}
\end{align}
From this, we can explicitly write $\mlij$ as follows
\begin{align}
    \la(e)\mlij(e)
    &=
    2 \sum_{k = 1}^N \int_\R e_k \Dk(\us) f(\us) \use ~dz\nonumber
    \\
    &=
    - \int_\R \partial_{e_i}(a_e)(\us) f(\us) 
    \left[
        \int_\a^\us
        \partial_{e_j}
        \left(
            \dfrac{ a_e(s)}{\sqrt{W_e(s)}}
        \right) ds
    \right]
    \usz ~dz\nonumber
    \\
    &=
    - \int_\am^\ap 
    \partial_{e_i} (a_e) (s) f(s) 
    \left[
        \int_\a^s
        \partial_{e_j}
        \left(
            \dfrac{ a_e(t)}{\sqrt{W_e(t)}}
        \right) dt
    \right]
    ds\nonumber
    \\
    &=
    \dfrac{1}{2} \int_\am^\ap 
    \partial_{e_i} (W_e(s))'
    \left[
        \int_\a^s
        \partial_{e_j}
        \left(
            \dfrac{ a_e(t)}{\sqrt{W_e(t)}}
        \right) dt
    \right]
    ds\nonumber
    \\
    &=
    - \dfrac{1}{2} \int_\am^\ap 
    \partial_{e_i} (W_e(s))
    \partial_{e_j}
    \left(
        \dfrac{ a_e(s)}{\sqrt{W_e(s)}}
    \right)
    ds.\label{eq:IM-2}
\end{align}

With \eqref{eq:IM-1} and \eqref{eq:IM-2} we are now ready to understand the well-posedness of the signed distance function. Using Theorem 2.1 of \cite{GG} it is enough to prove that 
\begin{align}\label{eq:mot-ellip}
    \sum_{i.j = 1}^d \tilde{\mu}_{ij}(e)  \eta_i \eta_j \ge C
    ,~
    \tilde\mu_{ij}(e) = \la(\msij(e) + \mlij(e))
\end{align}
for some positive constant $C$, where $\eta$ is a unit vector satisfying $e \cdot \eta = 0$. Namely, $(\tilde{\mu}_{ij})$ and therefore $(\mij)$ is non-degenerate to the tangential direction to the interface. Indeed, direct computation gives
\begin{align}
    \sum_{i,j = 1}^N 4 \tilde{\mu}_{ij}(e) \eta_i \eta_j
    &=
    \sum_{i,j = 1}^N \int_{\alpha_-}^{\alpha_+} 2 \partial_{e_i e_j}^2 a_e(s) \eta_i \eta_j \sqrt{W_e(s)}\nonumber
    \\
    &- \frac{\partial_{e_i} W_e(s)}{W_e(s)^{3/2}} \left( 2 \partial_{e_j} a_e(s) W_e(s) - a_e(s) \partial_{e_j} W_e(s) \right) \eta_i \eta_j \, ds\nonumber
    \\
    &=
    \sum_{i,j = 1}^N \int_{\alpha_-}^{\alpha_+} 2 \partial_{e_i e_j}^2 a_e(s) \eta_i \eta_j \sqrt{W_e(s)} \, ds\nonumber
    \\
    &+ \int_{\alpha_-}^{\alpha_+} \frac{1}{\sqrt{W_e(s)}}
    \left(
        -2 \overline{a}(s; e,\eta) \overline{W}(s; e,\eta)
        + \dfrac{a_e(s)}{W_e(s)} \overline{W}(s; e,\eta)^2
    \right) \, ds\nonumber
    \\
    &=
    \sum_{i,j = 1}^N \int_{\alpha_-}^{\alpha_+} 2 \partial_{e_i e_j}^2 a_e(s) \eta_i \eta_j \sqrt{W_e(s)} \, ds\nonumber
    \\
    &-
    \int_{\alpha_-}^{\alpha_+} W_e(s)^{-1/2} \dfrac{W_e(s)}{a_e(s)} \overline{a}(s; e,\eta)^2 \, ds\nonumber
    \\
    &+
    \int_{\alpha_-}^{\alpha_+} W_e(s)^{-1/2} \dfrac{a_e(s)}{W_e(s)} 
    \left(
        \dfrac{W_e(s)}{a_e(s)} \overline{a}(s; e,\eta) 
        - \overline{W}(s; e,\eta) 
    \right)^2 \, ds ,\label{eq:wp-1}
    \\
    \overline{a}(s;e,\eta) &:= \sum_{i = 1}^N \eta_i \partial_{e_i} a_e(s),~
    \overline{W}(s;e,\eta) := \sum_{i = 1}^N \eta_i \partial_{e_i} W_e(s).\nonumber
\end{align}
From the fact that $D(s) = (D_{ij})(s)$ is symmetric, for fixed $s$ we can find a diagonalization $\tilde{D}(s)$ of $D(s)$; thus there exists an orthonormal matrix $O(s)$ such that $D(s) = O(s) \tilde{D}(s) O(s)^t $ assume that $D(s)$ is a diagonal matrix by changing the axis; thus we can write $D(s) = (\tilde{D}_i(s))$. Let $\overline{e}(s;e), \overline{\eta}(s;\eta)$ be the vectors satisfying
\begin{align*}
    \sum_{i,j = 1}^N D_{ij}(s) e_i e_j = \sum_{i = 1}^N \tilde{D}_i \overline{e}_i^2
    ,~
    \sum_{i,j = 1}^N D_{ij}(s) \eta_i \eta_j = \sum_{i = 1}^N \tilde{D}_i \overline{\eta}_i^2.
\end{align*}
Thus, $\overline{e}, \overline{\eta}$ are the vectors equal to $e, \eta$ respectively but with different axis and satisfies $\overline{e} \cdot \overline{\eta} = 0$. This implies that
\begin{align*}
    \overline{a}(s;e,\eta)^2
    &= \left(\sum_{i = 1}^N 2 \tilde{D}_i(s) \overline e_i \overline\eta_i \right)^2
    =
    \left(\sum_{i = 1}^N 2 (\tilde{D}_i(s) - \underbar{D}(s)) \overline e_i \overline\eta_i \right)^2
    \\
    &\le
    4
    \left(\sum_{i = 1}^N (\tilde{D}_i(s) - \underbar{D}(s)) \overline{e}_i^2 \right)
    \left(\sum_{i = 1}^N (\tilde{D}_i(s) - \underbar{D}(s)) \overline\eta_i^2 \right)
    \\
    &\le
    4 a_e(s)
    \left( \sum_{i = 1}^N (\tilde{D}_i(s) - \underbar{D}(s)) \overline\eta_i^2 \right)
    \\
    \underbar{D}(s) &:= \min_{i = 1, \cdots d} \tilde{D}_i(s).
\end{align*}
Dropping the last term in \eqref{eq:wp-1}, we obtain
\begin{align*}
    \sum_{i,j = 1}^N 4 \tilde{\mu}_{ij}(e) \eta_i \eta_j
    &\ge
    \sum_{i,j = 1}^N \int_{\alpha_-}^{\alpha_+} 2 \partial_{e_i e_j}^2 a_e \eta_i \eta_j \sqrt{W_e(s)} ds
    \\
    &-
    \int_{\alpha_-}^{\alpha_+} W_e(s)^{-1/2} \dfrac{W_e(s)}{a_e(s)} \overline{a}(s; e,\eta)^2 ds
    \\
    &\ge
    4\sum_{i = 1}^N 
    \int_{\alpha_-}^{\alpha_+} 
    \left(
        \tilde{D}_i(s) \bar{\eta}_i^2
        - (\tilde{D}_i(s) - \underbar{D}(s)) \overline\eta_i^2
    \right) \sqrt{W_e(s)} ds
    \\
    &=4\sum_{i = 1}^N 
    \int_{\alpha_-}^{\alpha_+} \underbar{D}(s) \overline\eta_i^2 \sqrt{W_e(s)} ds,
\end{align*}
which leads to \eqref{eq:mot-ellip} since $\underbar{D}(s)$ is strictly positive in $[\alpha_-, \alpha_+]$. Thus, we obtain the following lemma for well-posedness of the interface $\Ga_t$ by using Theorem 2.1 of \cite{GG}.

\begin{lem}\label{lem:d0-WP}
    There exists a positive constant $T$ such that the solution $\Ga_t$ of \eqref{P0} exists uniquely in $[0,T]$ satisfying $\Ga_t \in C^{4 + \nu, 2 + \nu/2}$.
\end{lem}

\begin{rem}\label{rem:d0-p}
    The solution $\ul$ of \eqref{eqn:TW_lin,formal} used in the formal expansion is not well-defined. During the derivation of \eqref{eq:d-p} the derivative of the signed distance function $d_0$ was considered not only being a coefficient term, but also independent to the variable $z$. This may be true for the terms $\partial_t d_0$ and $\pvi d$ near the interface but not for the terms $\pvi \pvj d_0$, which leads to the fact that the solvability condition may fail away from the interface; see Proposition 2.2 of \cite{GG}. In the later section we will reintroduce the function $\ul$ satisfying the solvability condition \eqref{lem:TWlin_sol1} which will be important in the  proof of the main theorem.
\end{rem}


\section{Generation of the interface}\label{sec:generation}

In this section we prove the generation of the interface. Since we assumed that $\Vert u_0 \Vert_{C^2(\Om)}$ is bounded, studying the equation
\begin{align*}
    u_t 
    = 
    \dfrac{1}{\e^2} f(u),
\end{align*}
helps us to understand behavior of the equation \eqref{eq:pe} at least within a small time.
To be precise, as Theorem \ref{thm:gen} depicts, the generation occurs within the time scale of order $\o(\e^2 | \ln \e |)$ creating the steep transition layer which divides the steady states $\a_\pm$.

Under such intuition, we first consider the following ordinary differential equation
\begin{align}\label{eq:ODE_Y}
\begin{cases}
    Y_\tau(\tau;\xi)
    =
    f(Y)
    \\
    Y(0;\xi)
    =
    \xi.
\end{cases}
\end{align}
Recall $c_0$ in \eqref{cond:C0} and
\begin{align*}
     \eta_0 
     = \min(\ap - \a, \a - \am)
     ,~
     \nu = f'(\a)
\end{align*}
in Theorem \ref{thm:gen} and \eqref{cond:f-bistable}. We deduce the following result from \cite{AHM2008}.
\begin{lem}\label{lem:gen}
	Let $\eta \in (0, \eta_0)$ be arbitrary. Then, there exists a positive constant $C_Y 
	= C_Y(\eta)$ such that the following holds:
\begin{enumerate}[label =(\roman*)]
\item For all $\tau > 0$ and all $\xi \in (-2c_0, 2c_0)$,
\begin{align}\label{lem:gen_0}
	0 <
    Y_\xi(\tau,\xi)
	\leq 
	C_Y e^{\nu \tau}.
\end{align}

\item For all $\tau > 0$ and all $\xi \in (-2c_0, 2c_0)$,
\begin{align}\label{lem:gen_1}
    \left|
	\frac{Y_{\xi \xi}(\tau, \xi)}{Y_\xi(\tau, \xi)} 
	\right|
	\leq C_Y (e^{\nu \tau} - 1).
\end{align}

\item There exists a positive constant $\e_0$ such that, for all $\e \in (0, \e_0)$, we have
\begin{enumerate}
\item for all $\xi \in (-2c_0, 2c_0)$
\begin{align}\label{lem:gen_2}
	\am - \eta
	\leq
	Y(\nu^{-1} |\ln \e|, \xi)
	\leq
	\ap + \eta,
\end{align}

\item if $\xi \geq \a + C_Y \e$, then
\begin{align}\label{lem:gen_3}
	Y(\nu^{-1} |\ln \e|, \xi) \geq \ap - \eta,
\end{align}

\item if $\xi \leq \a - C_Y \e$, then
\begin{align*}
	Y(\nu^{-1} |\ln \e|, \xi) \leq \am + \eta.
\end{align*}
\end{enumerate}
\end{enumerate}

\end{lem}

We also give a comparison principle of \eqref{eq:pe}, which can be derived by using the maximum principle of semilinear parabolic differential equation; see \cite{Arena1972}.

\begin{lem}\label{lem:comparison}
Let $u^+$ be the functions satisfying 
\begin{align*}
    \begin{cases}
        \ml(u^+)
        \ge 0
        &\text{in}~ \Omega \times (0,T),
        \\
        \dfrac{\partial u^+}{\partial \nu} = 0
        &\text{on}~ \partial \Omega \times (0,T),
        \\
        u^+(x,0) \ge u_0(x)
        &x \in \Omega.
    \end{cases}
\end{align*}
And let $u^-$ be the function satisfying the opposite inequalities of the above equation. Then we have
\begin{align*}
    u^+ \ge u^- \text{in}~ \Omega \times (0,T).
\end{align*}
\end{lem}

With the help of these lemmas we now prove Theorem \ref{thm:gen}.

\begin{proof}[Proof of Theorem \ref{thm:gen}]
 We prove Theorem \ref{thm:gen} by constructing sub- and super-solutions.
\begin{align*}
    w^\pm(x,t)
    =
    Y
    \left(
        \dfrac{t}{\e^2}; u_0(x) \pm \e^2 P(t)
    \right)
    ,~
    P(t) = C_g \left( e^{\nu t/\e^2} - 1
    \right),
\end{align*}
where $C_g$ is a positive constant which will be defined later.

Here we show $w^+$ is a super-solution; one can show $w^-$ is a sub-solution in a similar way. And since  $u_0(x) \le w^+(x,0),~ x \in \Om$, we only need to prove $\ml w^+ \ge 0$. Direct computation gives
\begin{align*}
    w^+_t
    =
    \dfrac{Y_\tau}{\e^2}
    +
    \e^2 P'(t) Y_\xi
    ,~~
    \pvi w^+
    =
    Y_\xi \pvi u_0
    ,~~
    \pvi \pvj w^+
    =
    Y_{\xi\xi} \pvi u_0 \pvj u_0
    +
    Y_{\xi} \pvi \pvj u_0.
\end{align*}
Thus  by using \eqref{eq:ODE_Y} and Lemma \ref{lem:gen} we obtain
\begin{align*}
    \ml(w^+)
    &=
    \dfrac{Y_\tau }{\e^2} 
    + \e^2P'(t) Y_\xi
    - \D(Y) \pvi \pvj w^+
    - \D'(Y) \pvi w^+ \pvj w^+
    - \dfrac{f(Y)}{\e^2}
    \\
    &=
    Y_\xi
    \left(
        \e^2P'(t) 
        -
        \D(Y) 
        \left(
            \pvi u_0 \pvj u_0 \dfrac{Y_{\xi\xi}}{Y_\xi} 
            + \pvi \pvj u_0
        \right)
        -
        \D'(Y) \pvi u_0 \pvj u_0  Y_\xi
    \right)
    \\
    &\ge
    Y_\xi
    \left(
        C_g \nu e^{\nu t/\e^2} 
        - 
        C_D (2 c_0^2 C_Y e^{\nu t/\e^2} + c_0)  
    \right),
\end{align*}
where the inequality holds by \eqref{lem:gen_0} and \eqref{lem:gen_1}. Since $C_g$ is arbitrary, by  choosing $C_g$ large enough $w^+$ is a super-solution. 

We now prove the result of Theorem \ref{thm:gen} with $w^\pm$. Note that, choosing $0 < \e_0$ sufficiently small we have 
\begin{align*}
    0 \le P(t)  \le P(t^\e) = C_g(\e^{-1} -1) < c_0\e^{-2},
\end{align*}
which implies that
\begin{align*}
    u_0 \pm \e^2 P(t) \in (- 2c_0, 2 c_0).
\end{align*}
Hence by \eqref{lem:gen_2} we obtain \eqref{thm:gen_1}.

To prove \eqref{thm:gen_2} we use $w^-$. For this, we choose $M_0$ satisfying $M_0 \ge C_g + C_Y$. Then for $x \in \Om$ satisfying $u_0(x) \ge \a + M_0 \e$ we have
\begin{align*}
    u_0(x) - \e^2 P(t^\e)
    \ge
    \a + M_0 \e - C_g \e
    \ge
    \a + C_Y \e,
\end{align*}
thus by \eqref{lem:gen_3} we have \eqref{thm:gen_2}. By similar method we can also prove \eqref{thm:gen_3} using $w^+$. 

\end{proof}

\section{Motion of the interface} \label{sec:propagation}

In the previous section, we proved that the solution $u^\e$ generates a steep transition layer within a short time. In fact, combining the generation result with \eqref{cond:gamma0_normal} yields that the width of the steep transition layer is $\o(\e)$ which allows us to estimate $u^\e(x,t^\e)$ close to the steady states $\apm$ with $\eta_g$ error. For next step, we reduce this $\eta_g$ error in a small scale within a small time and show that the propagation of the interface is approximated by the motion equation \eqref{P0}.

In order to show this assertion, we construct a pair of suitable sub- and super- solutions $u^\pm(x,t)$ for the problem \eqref{eq:pe}. Following the intuition from Section \ref{sec:formal_asympt}, we intend to find a pair of sub- and super-solutions similar to the formal asymptotic expansion up to order $\e$:
\begin{align*}
    u^\e(x,t)
    \simeq
    \us\left( \frac{d(x,t)}{\e}; \nabla d \right)
    +
    \e \ul\left( \frac{d(x,t)}{\e},x,t; \nabla d \right),
\end{align*}
and satisfies 
\begin{align*}
    u^-(x,t^\e)
    \le
    u^\e(x,t^\e)
    \le
    u^+(x,t^\e),
\end{align*}
where $\us, \ul$ are solutions introduced in Section \ref{sec:formal_asympt}; recall $t^\e = \nu^{-1}\e^2 |\ln \e|$ given in Theorem \ref{thm:gen}. Then, by the comparison principle we obtain 
\begin{align*}
    u^-(x,t)
    \le
    u^\e(x,t)
    \le
    u^+(x,t)
\end{align*}
for $t^\e \le t \le T$.

To construct $u^\pm$ modifying the asymptotic expansion, we need some preparation related to the signed distance function $d_0$ and the linearized solution $U_1$. We explain these in the upcoming sections.

\subsection{Modified signed distance function}\label{sec:ctoff-d}

In this section we cut-off the signed distance function $d_0$ near the interface $\Ga_t$, for our analysis later. By Lemma \ref{lem:d0-WP} the signed distance function is well-defined. Moreover, it follows from Proposition 2.2 of \cite{GG} that there exists a positive constant $\tilde{d}_0$ such that $d_0(x,t)$ is smooth in the tublar neighborhood $\{ (x,t) \in \Om \times [0,T],~ |d_0(x,t)| \le 4 \tilde{d}_0 \}$ of $\Ga_t, t \in [0,T]$. Moreover, by choosing $\tilde{d}_0$ small enough we can also assume that 
\begin{align*}
    dist(\Ga_t, \partial \Om) \ge 4 \tilde{d}_0
    ~~
    \text{for all}
    ~t \in [0,T].
\end{align*}
Next, let $\rho(s)$ be a smooth increasing function on $\R$ such that 
\begin{align*}
    \rho(s)
    =
    \begin{cases}
        s
        &
        \text{if}~ |s| \le 2\tilde{d}_0,
        \\
        - 3 \tilde{d}_0
        &
        \text{if}~ s \le - 3 \tilde{d}_0,
        \\
        3 \tilde{d}_0
        &
        \text{if}~ s \ge 3 \tilde{d}_0.
    \end{cases}
\end{align*}
Then, we define the cut-off signed distance function $d$ by
\begin{align*}
    d(x,t) = \rho \left( d_0(x,t) \right).
\end{align*}
Note that, since $d_0 = d$ near $\Ga_t$ and constant away from $\Ga_t$ we have 
\begin{align*}
    |\nabla d| = 1 ~\text{in} ~\{ (x,t) \in \Om \times [0,T],~ |d_0| \le 2 \tilde{d}_0  \},
    \\
    |\nabla d| = 0 ~\text{in} ~\{ (x,t) \in \Om \times [0,T],~ |d_0| \ge 3 \tilde{d}_0  \}.
\end{align*}    
In addition, the equation \eqref{eq:d0-p} also holds for $d$ on the interface $\Ga_t$ as well, thus satisfying 
\begin{align} \label{eq:d-p}
    \partial_t d 
    = \mij(\nabla d)  \pvi \pvj d
    ~\text{on}
    ~\Ga_t,
\end{align}
where we omitted the summation $\sum_{i,j = 1}^N$ and the coefficient $\mij$ is a function on $\mathbb{S}^{N - 1}$. We also give a lemma that will be used in the proof later.

\begin{lem}\label{lem:d}
    There exists a positive constant $C_d$ such that 
    \begin{enumerate}[label =(\roman*)]
        \item $\Vert d \Vert_{C^{4+\nu,2+\nu/2}(\Om \times [0,T])} \le C_d$,
        
        \item 
        $\left| \partial_t d - \sum_{i,j = 1}^N \mij(\nabla d)  \pvi \pvj d \right| \le C_d |d|$ in $\Om \times [0,T]$.
        
    \end{enumerate}
\end{lem}

\begin{proof}

The result $(i)$ is a direct consequence of Proposition 2.2 of \cite{GG}. And this result implies that the terms $d_t, \pvi d, \pvi \pvj d$ and $\mij$ are all Lipschitz continuous. Thus, the result $(ii)$ holds, since by \eqref{eq:d0-p} we have 
\begin{align*}
    \partial_t d - \sum_{i,j = 1}^N \mij(\nabla d)  \pvi \pvj d = 0
\end{align*}
on $\{ (x,t) \in \Om \times [0,T], d(x,t) = 0 \} $.
    
\end{proof}

\subsection{Estimates of $\us$ and linearized solution  $\ul$ }
 
In this section we give estimates related to $\us$ and $\ul$. We first give estimates on the solution $\us$ of  \eqref{eq:TW-e}.
\begin{lem}\label{lem:TW-est}
    There exist positive constants $C_0, {\lambda}_0$ such that 
    \begin{align}\label{lem:TW-est1}
        \begin{cases}
        0 < \ap - \us \le C_0 e^{- \lambda_0 |z|}
        ,~
        &\text{for}~ z \ge 0,
        \\
        0 < \us - \am \le C_0 e^{-\lambda_0 |z|}
        ,~
        &\text{for}~ z \le 0,
        \end{cases}
    \end{align}
    and
    \begin{align*}
        0 < \usz \le C_0 e^{- \lambda_0 |z|}
        ,~
        \left| \partial_z^{k_z} \pei^{k_i} \us \right| \le C_0 e^{- \lambda_0 |z|}
    \end{align*}
    for all $0 \le i \le N, (z;e) \in \R \times \mathbb{S}^{N - 1}$, where $k_z, k_i \in \mathbb{Z}^+, \, k_z + \sum_{i = 1}^N k_i \le 2$.
\end{lem}
\begin{proof}
We first prove the result for a fixed $e \in \mathbb{S}^{N - 1}$ then we can find the desired result since $\us, \usz, \uszz$ are continuous in $e$ and $\mathbb{S}^{N - 1}$ is compact. Let $V_0 := A_e(\us)$, where $A'_e(s) = a_e(s) > 0$ by \eqref{cond:D_pos}. Then from \eqref{eq:TW-e} we obtain
\begin{align*}
    \begin{cases}
        V_{0zz} + g(V_0) = 0
        \\
        V_0(\pm\infty) = \apm'
        ,~
        V_0(0) = \a'
    \end{cases}
\end{align*}
where $g(s) = f (A_e^{-1}(s)),~ \apm' = A_e(\apm),~ \a' = A_e(\a)$. Then by Lemma 2.1 of \cite{AHM2008} we can show the desired result except the boundedness of $ \Vert \us(z;\cdot) \Vert_{C^2(\mathbb{S}^{N - 1})}$  for any $z \in \R$. We start from \eqref{eq:TW_lin_explicit}.  By \eqref{cond:f-bistable} and \eqref{cond:D_pos} one can say that 
    \begin{align*}
        W_e(s) \le C_W
        ,~~
        C_W^{-1} (s - \a_-)^2 (\a_+ - s)^2
        \le
        W_e(s)
        \le
        C_W (s - \a_-)^2 (\a_+ - s)^2,
    \end{align*}
    for every $e \in \mathbb{S}^{N - 1}$, where $C_W$ is some positive constant. This implies that 
    \begin{align*}
        \left|
        \int_\a^s
        \partial_{e_i}
        \left(
            \dfrac{ a_e(t)}{\sqrt{W_e(t)}}
        \right)
        dt
        \right|
        +
        \left|
        \int_\a^s
        \pei \pej
        \left(
            \dfrac{ a_e(t)}{\sqrt{W_e(t)}}
        \right)
        dt
        \right|
        &\le 
        \tilde{C}_W 
        \left| 
        \ln (s - \a_-) 
        \right | 
        \\
        & +
        \tilde{C}_W \left| \ln(\a_+ - s)
        \right|
    \end{align*}
    for  every $e \in \mathbb{S}^{N - 1}$ and $1 \le i,j \le N$, where $\tilde{C}_W$ is some positive constant. Moreover, from \eqref{eq:tw_1-2} we can derive that 
    \begin{align*}
        \usz
        \le
        c_W (\us - \a_-)(\a_+ - \us)
    \end{align*}
    for  every $e \in \mathbb{S}^{N - 1}$, where $c_W$ is some positive constant. Thus we obtain
    \begin{align*}
        |\pei \us|
        \le
        \tilde{c}_W (\us - \a_-)(\a_+ - \us)
        (| \ln (\us - \a_-) | + | \ln(\a_+ - \us) | )
    \end{align*}
    for  every $e \in \mathbb{S}^{N - 1}$, where $\tilde{c}_W$ is some positive constant. Also, from direct computations we can also obtain that
    \begin{align*}
        |\pei \usz|
        &\le
        \tilde{c}_W (\us - \a_-)(\a_+ - \us) (| \ln (\us - \a_-)| + |\ln(\a_+ - \us) | )
        ,\\
        |\pei \pej \us|
        &\le
        \tilde{c}_W (\us - \a_-)(\a_+ - \us) \left(| \ln (\us - \a_-)|^2 + |\ln(\a_+ - \us) |^2\right),
    \end{align*}
    by choosing $\tilde{c}_W$ larger if needed. Therefore by \eqref{lem:TW-est1} we obtain the desired result.
    
\end{proof}

For $\ul$, as discussed in the Remark \ref{rem:d0-p} we need a different $G(z,x,t)$ of \eqref{eq:TW-lin} instead of the one used in \eqref{eqn:TW_lin,formal}. For this purpose, we define $\ul(z,x,t;e)$ as a solution satisfying the following ordinary differential equation:
\begin{align}
    \begin{cases}
        (a_e(\us)\ul)_{zz} + f'(\us) \ul
        =
        \mathcal{G}(z,x,t;e)
        ,~~z \in \R,~e \in \mathbb{S}^{N - 1}
        \\
        \ul(0;e) = 0
        ,~~\ul(\cdot;e) \in L^\infty(\R).
    \end{cases}\label{eq:TW-lin_subsuper}
\end{align}
Here $\mathcal{G}(z,x,t;e)$ is a function defined by
\begin{align*}
    \mathcal{G}(z,x,t;e)
    &=
    [
        (\msij(e) \usz - \D(\us) \usz)
    \\
    &
        +
        (\mlij(e) \usz - ( \pei(a_e)(\us) \use )_z)
    ]\pvi \pvj d,
\end{align*}
where we omitted the summation $\sum_{i,j = 1}^N$. Note that we replaced $d_0$ in \eqref{eqn:TW_lin,formal} by the cutoff signed distance function $d$. Moreover, as  $\sum_{i,j = 1}^N \mij(\nabla d) \pvi \pvj d$ is close to $\partial_t d$ in view of \eqref{eq:d-p} and Lemma \ref{lem:d}, we replaced $\partial_t d_0$ in \eqref{eqn:TW_lin,formal} to $\sum_{i,j = 1}^N \mij(e) \pvi \pvj d$. Due to the definitions of $\msij(e)$ and $\mlij(e)$ the function $\mathcal{G}$ now satisfies the condition \eqref{lem:TWlin_sol1} independent to the choice of $(x,t)$. We also give estimates of $\ul$ which will be needed later.
\begin{lem}\label{lem:TWlin_est}
There exists positive constants $C_1, \la_1$ such that 
\begin{align*}
    |\partial_t \ul|
    +
    \left| \partial_z^{k_z} \pvi^{k_i} \pej^{k_j} \ul \right|  
    \le 
    C_1 e^{- \la_1 |z|}
\end{align*}
for all $1 \le i,j \le N, (z,x,t;e) \in \R \times \Om \times [0,T] \times \mathbb{S}^{N - 1}$, where 
\begin{align*}
    k_z, k_i, k_j \in \mathbb{Z}^+,\, k_z + \sum_{i = 1}^N  k_i + \sum_{j = 1}^N  k_j \le 2.
\end{align*}
\end{lem}

\begin{proof}
The boundedness of derivatives with respect to $z$, $x$ and $t$ are guaranteed by Lemma \ref{lem:d} and  \cite{EFHPS-2}. Thus we focus on the boundedness of derivatives with respect to $e$. For this, by noting that $\pvi \ul$ satisfies the equation \eqref{eq:TW-lin} with
\begin{align*}
    G(z,x,t;e)
    =
    \pvi \mathcal{G}(z,x,t;e)
    -
    (\pvi(a_e(\us))\ul)_{zz}
    -
    \pvi(f'(\us))\ul,
\end{align*}
one can use the same reasoning as above to show the desired result.
\end{proof}

\subsection{Construction of sub- and super-solutions}

In this section, we construct a pair of sub- and super-solutions using $\us$ and $\ul$. We construct our sub-and super-solutions $u^\pm$ modifying $\tilde{u}^\pm$ in the form
\begin{align*}
    \tilde{u}^\pm
    =
    \us
    \left(
        z_d; \nabla d
    \right)
    +
    \e \ul 
    \left(
        z_d, x, t; \nabla d
    \right),
\end{align*}
where $z_d \simeq d(x,t)/\e$ and we will define later. However, this form is well-defined only when $\nabla d \in \mathbb{S}^{N - 1}$; thus $\tilde{u}^\pm$ are defined only near the interface $\Ga_t$ within the distance $2 d_0$. In order to define the sub- and super-solutions also away from the interface, we cut-off the function $\tilde{u}^\pm$. Similar to the function used in Section \ref{sec:ctoff-d} choose a smooth function $\rho_i(s), i= 1,2$  on $\R$ such that $0 \le \rho_1 \le 1$ and
\begin{align*}
    \rho_1(s)
    &=
    \begin{cases}
        0
        &
        \text{if}~ |s| \le \tilde{d}_0,
        \\
        1
        &
        \text{if}~ |s| \ge 2 \tilde{d}_0,
    \end{cases}
    \\
    \rho_2(s)
    &=
    \begin{cases}
        \a_+
        &
        \text{if}~ s \ge \tilde{d}_0,
        \\
        \a_-
        &
        \text{if}~ s \le - \tilde{d}_0.
    \end{cases}
\end{align*}
Then we define our sub- and super-solutions $u^\pm$ as follows;
\begin{align*}
    u^\pm
    =
    (1 - \rho_1(d)) \tilde{u}^\pm  + \rho_1(d) \rho_2(d) \pm q(t)
\end{align*}
where 
\begin{align}
    z_d(x,t)
    &=
    \dfrac{d(x,t) \pm \e p(t)}{\e}\label{eq:zd}
    ,\\
    p(t) 
    &= - e^{- \beta t/\e^2} + e^{L t} + K\nonumber
    ,\\
    q(t) 
    &= \sigma(\beta e^{- \beta t/\e^2} + \e^2 L e^{Lt} ).\nonumber
\end{align}
Here $\sigma, \beta, L$ and $K$ are positive constants which will be defined later. In addition we assume $0 < \e_0 < 1$ small enough such that 
\begin{align}\label{cond:q_bound}
    \e_0 p(t)
    \le
    \tilde{d}_0/2
    ,~~
    |\e_0 \ul| + q(t)
    \le
    \e_0 C_1 
    +
    \sigma(\beta + \e_0^2 L e^{LT})
    \le
    \eta_0
    ,~~
    L \e_0^2 e^{LT} < 1.
\end{align}

Constructed functions $u^\pm$ are composed of mainly 3 terms; $\us, \ul$ and $q$. Each of the terms has important purpose in making $u^\pm$ as sub- and super-solutions. As we discussed in Section \ref{sec:formal_asympt}, the function $\us(z_d;\nabla d)$ helps us to describe the steep transition layer connecting the stable steady states $\a_\pm$ and the function $\ul(z_d,x,t; \nabla d)$ helps us to describe the motion equation.
The term $q(t)$ helps us to make the constructed functions $u^\pm$ to be an actual sub- and super- solutions. Intuitively, since $\tilde{u}^\pm$ are expected to be close to the actual solution $u^\e$, the term $\pm q$ adjusts the function $\tilde{u}^\pm$ thereby giving an upper and lower bound of $u^\e$. 
Note that, the scale of $q$ changes as time goes. In the beginning, $q$ has of scale $\o(1)$ and decreases exponentially fast towards the scale of $\o(\e^2)$. To distinguish this scale to others we denote scales related to $q$ as $\o(q)$.

We give the following lemma for $u_\pm$.

\begin{lem}\label{lem:Prop_subsuper}
For any $K> 1$ there exist large enough $L> 0$ and small enough $0 < \sigma,  \e_0 < 1$ such that 
\begin{align}\label{eqn_Prop_subsuper}
	\begin{cases}
		\mathcal{L} (u^-) \leq 0 \leq \mathcal{L}(u^+)
		&
		\text{ in } \Om \times [0,T - t^\e]
		\\
		\displaystyle{\frac{\partial u^-}{\partial \nu}
		= \frac{\partial u^+}{\partial \nu}}
		= 0
		&
		\text{ on } \partial \Om \times [0,T - t^\e]
	\end{cases}
\end{align}
for every $\e \in (0, \e_0)$.
\end{lem}

\begin{proof}
    From this proof we denote $e$ as $\nabla d$. 
    We prove that $\ml(u^+) \ge 0$; by similar method one can prove also $\ml(u^-) \le 0$. Due to the cut-off in the solution $u^+$ we divide the case into three.

    \begin{enumerate}
    \item In the set $\Om'_T := \{ (x,t) \in \Om \times [0,T - t^\e],~ |d(x,t)| \le \tilde{d}_0\}$

    To show the assertion it is necessary to compute $\ml(u^+)$ directly. For this, we preform a similar computation as in Section \ref{sec:formal_asympt}; (1) Taylor expansion of the nonlinear terms such as $\D$ and $f$ and (2) direct computation of the derivatives. 
    We first preform the Taylor expansion, where we obtain 
    \begin{align}\label{eq:taylor}
    \begin{cases}
       \D(\us + \vp) = \D(\us) + \D'(\us) \vp + \dfrac{\D''(\theta_1(x,t))}{2}\vp^2,
        \\
        \D'(\us + \vp) = \D'(\us) + \D''(\us) \vp + \dfrac{\D'''(\theta_2(x,t))}{2}\vp^2,
        \\
        ~~~f(\us + \vp) = f(\us) + f'(\us) \vp + \dfrac{f''(\theta_3(x,t))}{2}\vp^2.
    \end{cases}
    \end{align}
    Here $\vp = \e \ul + q$ and $\theta_i$ are some constants between $\us$ and $\us + \vp$. We can divide the terms into 3 groups. (1) Terms only related to $\us$ such as $\D(\us), \D'(\us)$ and $f(\us)$, (2) terms related to $\e \ul$ and (3) terms related to $q$. Each of them represents the terms of scale $\o(1)$, $\o(\e)$  and $\o(q)$ respectively. 

    Next we preform direct computation of the derivatives. By noting that the $\o(\e^{-1})$ scale appears by taking derivatives with respect to $z_d$ as in \eqref{eq:zd}, one can see that $\o(\e^{-2})$ terms appear by taking derivative twice to the term $\us$ with respect to $z_d$, and $\o(\e^{-1})$ terms appear by taking derivative twice to the term $\e\ul$ with respect to $z_d$ or taking derivative one time to the term $\us$ with respect to $z_d$. Thus we obtain the following computations
    \begin{align}
        u_t^+
        =&
        \left(
            \usz + \e \ulz
        \right)
        \dfrac{d_t}{\e}
        +
        \partial_e( \us + \e \ul ) \cdot \nabla d_t
        + \e U_{1t}
        + q'\label{eq:deri_t}
        ,\\
        \pvi u^+
        =&
        \left(
            \usz + \e \ulz
        \right)
        \dfrac{\pvi d}{\e}
        +
        \partial_e( \us + \e \ul ) \cdot \pvi \nabla d 
        + \e \pvi \ul \label{eq:deri_x0}
        ,\\
        \pvi \pvj u^+
        =&
        \dfrac{\pvi d \pvj d}{\e^2} \uszz
        +
        \dfrac{r_{11ij}}{\e}
        + r_{12ij}(x,t)\label{eq:deri_x1}
        ,\\
        \pvi u^+ \pvj u^+
        =&
        \dfrac{\pvi d \pvj d}{\e^2} \usz^2
        +
        \dfrac{r_{21ij}}{\e}
        + r_{22ij}(x,t)\label{eq:deri_x2}
        ,
        \end{align}
        where
        \begin{align*}
        r_{11ij}
        =&
        \ulzz \pvi d \pvj d
            +
            \usz \pvi \pvj d
            +
            \pe \usz \cdot (\pvi d \pvj \nabla d + \pvj d \pvi \nabla d)
        ,\\
        r_{12ij}
        =&
        \ulz{\pvi \pvj d}  
        +
        \pe \ulz \cdot (\pvi d \pvj \nabla d + \pvj d \pvi \nabla d)
        +
        \pvi \ulz \pvj d
        \\
        +&
        [\partial_e^2(\us + \e \ul) \pvi \nabla d  
        +
        \e \partial_e\pvi \ul ] 
        \cdot \pvj \nabla d
        +
        \partial_e(\us + \e \ul) \cdot \pvi \pvj \nabla d
        \\
        +&
        \pvj \ulz \pvi d
        +
        \e \partial_e \pvj \ul \cdot \pvi \nabla d
        +
        \e \pvi \pvj \ul
        ,\\
        r_{21ij}
        =&
        2 \usz \ulz \pvi d \pvj d
            +
            \usz \pe \us \cdot (\pvi d \pvj \nabla d + \pvj d \pvi \nabla d)
        ,\\
        r_{22ij}
        =&
        \ulz^2 \pvi d \pvj d 
        \\
        +&
        (\usz \pe \ul + \ulz \pe \us + \e \ulz \pe \ul)\cdot (\pvi d \pvj \nabla d + \pvj d \pvi \nabla d)
        \\
        +&
        (\usz + \e \ulz) (\pvi d \pvj \ul + \pvj d \pvi \ul)
        \\
        +&
        (\partial_e(\us + \e \ul) \cdot \pvi \nabla d + \e \pvi \ul)
        (\partial_e(\us + \e \ul) \cdot \pvj \nabla d + \e \pvj \ul)
        \end{align*}
    Here the terms $\e^{-1}r_{11ij}, \e^{-1}r_{21ij}$ are $\o(\e^{-1})$ scale terms and $r_{12ij}, r_{22ij}$ are $\o(1)$ scale terms. Since $r_{11ij}, r_{12ij}, r_{21ij}, r_{22ij}$ consists of derivatives of $\us$ and $\ul$, by Lemmas \ref{lem:d}, \ref{lem:TW-est} and \ref{lem:TWlin_est}  there exists a positive constant $C_r$ such that 
    \begin{align}\label{eq:r_bounded}
        |r_{11ij}(x,t)| + |r_{12ij}(x,t)|
        +
        |r_{21ij}(x,t)| + |r_{22ij}(x,t)|
        \le C_r e^{- \tilde{\la} |z_d|},
    \end{align}
    in $\Om'_T$ and for every $1 \le i,j \le N$, where $\tilde{\la} = \min \{ \la_1, \la_2 \}$. Also, in a similar reason we can also say that 
    \begin{align}\label{eq:r_bounded_1}
        \left|\pvi u^+ \pvj u^+ \right| 
        + 
        \left|\pvi \pvj u^+ \right| 
        \le
        \dfrac{C_r}{\e^2} e^{- \tilde\la |z_d|},
    \end{align}
    by letting $C_r$ larger if needed. Note that such $C_r$ can be chosen independent to the construction of $u^+$.
    
    Combining these we first compute the leading terms $\o(\e^{-2})$ and $\o(\e^{-1})$ in
    \begin{align*}
        \pvi(\D(u^+) \pvj u^+) 
        =
        \D(u^+) \pvi \pvj u^+ 
        +
        \D'(u^+) \pvi u^+ \pvj u^+.
    \end{align*}
    To obtain the $\o(\e^{-2})$ scale terms we need to multiply the $\o(1)$ scale terms of \eqref{eq:taylor} and $\o(\e^{-2})$ scale terms of \eqref{eq:deri_x1} and \eqref{eq:deri_x2}, which gives
    \begin{align*}
        \sum_{i,j = 1}^N (\D(\us) \uszz + \D'(\us) \usz^2)\pvi d \pvj d
        &=
        \sum_{i,j = 1}^N (\D(\us) \usz)_z \pvi d \pvj d
        \\
        &=
        (a_e(\us) \usz)_z
        =
        (A_e(\us))_{zz},
    \end{align*}
    where the equality holds since $e_i = \pvi d$ is independent to $z$.
    To obtain the $\o(\e^{-1})$ scale terms we multiply 
    (i) $\o(1)$ scale terms of \eqref{eq:taylor} and $\o(\e^{-1})$ scale terms of \eqref{eq:deri_x1} and \eqref{eq:deri_x2}, that is, $\D(\us) r_{11ij} + \D'(\us) r_{21ij;}$
    (ii) $\o(\e)$ scale terms of \eqref{eq:taylor} and $\o(\e^{-2})$ scale terms of \eqref{eq:deri_x1} and \eqref{eq:deri_x2}(as we mentioned above, $q$ is excluded here), which gives
    \begin{align*}
        (\D(\us) \ulzz + 2\D'(\us) \usz \ulz)\pvi d \pvj d 
        &=
        a_e(\us)\ulzz + 2 a_e'(\us) \usz \ulz
        \tag{i-1}
        \\
        &=
        a_e(\us)\ulzz + 2 a_e(\us)_z \ulz
        \\
        \D(\us) \usz \pvi \pvj d\tag{i-2}
        \\
        \D(\us) \pe \usz \cdot (\pvi d \pvj \nabla d + \pvj d \pvi \nabla d)
        &=
        2\D(\us) \pe \usz \cdot \pvi d \pvj \nabla d \tag{i-3}
        \\
        &=
        \pei(a_e)(\us) \pej \usz \pvi \pvj d
        \\
        \D'(\us) \usz \pe \us  \cdot (\pvi d \pvj \nabla d + \pvj d \pvi \nabla d)
        &=
        2\D'(\us) \usz \pe \us \cdot \pvi d \pvj \nabla d \tag{i-4}
        \\
        &=
        \pei(a'_e)(\us) \usz \pej \us  \pvi \pvj d
        \\
        (\D'(\us) \uszz  + \D''(\us) \usz^2) \ul \pvi d \pvj d \tag{ii}
        &=
        (a_e'(\us)\usz)_z \ul
        \\
        &=
        (a_e(\us))_{zz} \ul,
    \end{align*}
    where we omitted the summation $\sum_{i,j = 1}^N$. Here the first equalities of (i-3) and (i-4) holds since $\D$ is symmetric. Also, combining the computations (i-1) and (ii) gives
    \begin{align*}
        \text{(i-1)} + \text{(ii)}
        =
        (a_e(\us)\ul)_{zz},
    \end{align*}
    and combining (i-3) and (i-4) gives
    \begin{align*}
        \text{(i-3)} + \text{(i-4)}
        =
        (\pei(a_e)(\us) \use)_z \pvi \pvj d.
    \end{align*}

    With computations above, we can write $\sum_{i,j = 1}^N \pvi(\D(u^+) \pvj u^+)$  as follows
    \begin{align*}
        \sum_{i,j = 1}^N
        \pvi (\D(u^+) \pvj u^+)
        =&
        \dfrac{A_e(\us)_{zz}}{\e^2}
        + \dfrac{(a_e(\us)\ul)_{zz}}{\e}
        \\
        +& \dfrac{\D(\us)\usz + (\pei(a_e)(\us) \use )_z}{\e}\pvi \pvj d
        \\
        +& (\D'(\us) \pvi \pvj u^+ + \D''(\us) \pvi u^+ \pvj u^+)q
        \\
        +&\dfrac{(\e \ul + q)^2}{2}
        \left[
            \D''(\theta_1) \pvi \pvj u^+  
            + \D'''(\theta_2) \pvi u^+ \pvj u^+ 
        \right]
        \\
        +& R(x,t)
        ,
        \\
        R(x,t)
        =&
        \D'(\us) \ul r_{11ij}
        +
        \D(u^+) r_{12ij} 
        \\
        +&
        \D''(\us) \ul r_{21ij}
        +
        \D'(u^+) r_{22ij} 
    \end{align*}
    where we omitted the summation $\sum_{i,j = 1}^N$ in the right hand side. Note that except the term $R$, we wrote every $\o(\e^{-2})$, $\o(\e^{-1})$ scale terms and terms multiplied with $q$. Since $D$ is assumed to be $C^3$ and bounded by \eqref{cond:D_pos}, by \eqref{eq:r_bounded} there exists a positive constant $C_R$ independent to the construction of $u^+$ such that 
    \begin{align}\label{eq:R_bounded}
        |R(x,t)| \le C_R e^{- \tilde{\la} |z_d|},
    \end{align}
    holds. With this, \eqref{eq:taylor} for $f(\us + \vp)$ and \eqref{eq:deri_t} for $u^+_t$, we can divide the terms of $\ml(u^+)$ as follows
    \begin{align*}
        \ml(u^+)
        =
        E_1 + \cdots + E_5
    \end{align*}
    where
    \begin{align*}
        E_1
        =&
        - \dfrac{A_e(\us)_{zz} + f(\us)}{\e^2}
        ,\\
        E_2
        =&
        \dfrac{\usz}{\e}
        d_t
        \\
        -&
        \dfrac{1}{\e}
        \left[
            (a_e(\us)\ul)_{zz} + f'(\us) \ul
        \right.
        \\
        &\left.
            + (\D(\us) \usz  + (\pei(a_e)(\us) \use)_z)\pvi \pvj d
        \right]
        ,\\
        E_3
        =&
        {\ulz}d_t
        +
        \partial_e( \us + \e \ul ) \cdot \nabla d_t
        +
        \e U_{1t}
        - R
        ,\\
        E_4
        =&
        \usz p' + q' 
        - (\D'(\us) \pvi \pvj u^+ + \D''(\us) \pvi u^+ \pvj u^+)q
        \\
        -& f'(\us)\dfrac{q}{\e^2} 
        + 
        \e \ulz p'
        ,\\
        E_5
        =&
        \dfrac{(\e \ul + q)^2}{2}
        \left[
            \D''(\theta_1) \pvi \pvj u^+  
            + \D'''(\theta_2) \pvi u^+ \pvj u^+ 
            + \dfrac{f''(\theta_3)}{\e^2}
        \right]
        .
    \end{align*}
    The terms $E_i$ are gathered in the following way; $E_1$, $E_2$ and $E_3$ are composed of the terms of order $\o(\e^{-2})$, $\o(\e^{-1})$ and $\o(1)$ respectively except the terms with $p$ and $q$, $E_4$ composed of the terms multiplied with $p$ and $q$ and $E_5$ composed of the terms multiplied with $(\e \ul + q)^2$. 

    \begin{enumerate}[label =(\roman*)]
        \item The term $E_1$. By \eqref{eq:TW-e} we have
        \begin{align*}
            E_1 = 0.
        \end{align*}

        \item The term $E_2$. By \eqref{eq:TW-lin_subsuper} we have
        \begin{align*}
            E_2 = \dfrac{d_t - \mij(e) \pvi \pvj d}{\e} \usz,
        \end{align*}
        where we omitted the summation $\sum_{i,j = 1}^N $. Then, by Lemmas \ref{lem:d} and \ref{lem:TW-est} we obtain
        \begin{align*}
            |E_2| 
            \le &
            \dfrac{C_d |d|}{\e} \usz
            \le
            C_d C_0 (p(t) + |z_d|) e^{-\la_0|z_d|}
            \\
            \le&
            C_d C_0 (e^{Lt} + K + |z_d|) e^{-\la_0|z_d|}
            \\
            \le&
            S_2 e^{Lt},
        \end{align*}
        for some positive constant $S_2$, where the last inequality holds since $|z| e^{-\la_0|z|}$ is bounded in $\R$.

        \item The term $E_3$. 
        By Lemmas \ref{lem:d}, \ref{lem:TW-est}, \ref{lem:TWlin_est} and \eqref{eq:R_bounded} we obtain
        \begin{align*}
            |E_3|
            \le
            S_3,
        \end{align*}
        for some positive constant $S_3$.

        \item The term $E_4$. In view of \eqref{eq:deri_x1} and \eqref{eq:deri_x2}, the $\o(\e^{-2})$ scale leading term of $\D'(\us) \pvi \pvj u^+ + \D''(\us) \pvi u^+ \pvj u^+$ is
        \begin{align*}
            \sum_{i,j = 1}^N
            (\D'(\us) \uszz + \D''(\us) \usz^2) \pvi d \pvj d
            =
            (a_e(\us))_{zz}.
        \end{align*}
        Let
        \begin{align*}
            \tilde{R}(x,t) : = \D'(\us) \pvi \pvj u^+ + \D''(\us) \pvi u^+ \pvj u^+ - \e^{-2}(a_e(\us))_{zz}.
        \end{align*}
        Then, \eqref{eq:deri_x1} and \eqref{eq:deri_x2} show
        \begin{align*}
            \tilde{R}(x,t)
            =
            \D'(\us)
            \left(
                \dfrac{r_{11ij}}{\e} + r_{12ij}
            \right)
            +
            \D''(\us)
            \left(
                \dfrac{r_{21ij}}{\e} + r_{22ij}
            \right)
        \end{align*}
        where we omitted the summation $\sum_{i,j = 1}^N$. Thus,  by \eqref{eq:r_bounded} we have
        \begin{align*}
            \e^2 |\tilde{R}| \le \e \tilde{C}_R,
        \end{align*}
        for some positive constant $\tilde{C}_R$.
        
        Then, noting that $q = \e^2 \sigma p'$ and recalling the definitions of $\tilde{R}(x,t)$ and $q(t)$ we have
        \begin{align*}
            E_4
            =&
            \dfrac{q}{\e^2}
            \left[
                \dfrac{\usz}{\sigma} 
                - [a_e(\us)_{zz}  + f'(\us)]
                - \e^2 \tilde{R}
            \right]
            + q'
            + 
            \e \ulz p'
            \\
            =&
            \sigma\dfrac{\beta e^{-\beta t/\e^2}}{\e^2}
            \left[
                \dfrac{\usz}{\sigma} 
                - [a_e(\us)_{zz}  + f'(\us) + \e^2 \tilde{R}]
                - \beta
            \right]
            \\
            +&
            \sigma L e^{Lt}
            \left[
                \dfrac{\usz}{\sigma} 
                - [a_e(\us)_{zz} + f'(\us) + \e^2 \tilde{R}]
                + \e^2 L
            \right]
            + 
            \dfrac{\e \ulz}{\sigma \e^2} q.
        \end{align*}
        The fact that for any $e \in \mathbb{S}^{N - 1}$, $a_e(\us(z;e))_{zz}$ converges to 0 as $z \to \pm \infty$ by Lemma \ref{lem:TW-est} and that $f'(\us(z;e)) < 0$ for $|z|$ large enough by \eqref{cond:f-bistable} and \eqref{lem:TW-est1} implies that, choosing $Z_0 > 0$ large enough we can find a positive constant $B$ such that 
        \begin{align}\label{eq:f'bound}
            - [a_e(\us)_{zz} + f'(\us)] > 4B
        \end{align}
        for $|z| \ge Z_0$. Moreover, since $\usz > 0$ in $\R$, choosing $0 < \sigma < 1$ small enough we have
        \begin{align*}
            \dfrac{\usz}{\sigma} 
            - [a_e(\us)_{zz} + f'(\us)]
            \ge 4B,
        \end{align*}
        for $|z| \le |Z_0|$. Choose $\beta \in (0, B)$ and $\e_0 > 0$ small enough such that
        \begin{align}\label{eq:e0sigma_cond}
            \tilde{C}_R \e_0 \le B
            ,~
            \e_0 C_1 \le \sigma B,
        \end{align}
        where $C_1$ is a constant appeared in Lemma \ref{lem:TWlin_est}. 
        With this, we can derive that there exists a positive constant $S_4$(indeed, equal to $B$) satisfying
        \begin{align*}
            E_4 
            \ge &
            2\sigma\dfrac{\beta e^{-\beta t/\e^2}}{\e^2} B
            +
            3\sigma L e^{Lt} B
            - B\dfrac{q}{\e^2}
            \\
            \ge&
            2B \dfrac{q}{\e^2}
            - B\dfrac{q}{\e^2}
            \ge
            S_4 \dfrac{q}{\e^2}.
        \end{align*}

        \item The term $E_5$. 
        Note that the terms $\pvi \pvj u^+, \pvi u^+ \pvj u^+$ are $\o(\e^{-2})$ scale by \eqref{eq:r_bounded_1}. This implies that one can find a positive constant $\tilde{B}$ 
        \begin{align*}
            |E_5| 
            \le
            \tilde{B} 
            \left(
                \ul^2 + 2 \ul \dfrac{q}{\e} + \dfrac{q^2}{\e^2}
            \right).
        \end{align*}
        And, from \eqref{cond:q_bound} we can derive that
        \begin{align*}
            q \le \sigma(\beta + \e^2 L e^{LT})
            \le \sigma(\beta + 1).
        \end{align*}
        By Lemma \ref{lem:TWlin_est}, \eqref{eq:e0sigma_cond} and $q \le \sigma (\beta + 1)$, we see that
        \begin{align}\label{eq:(eu1+q)^2bound}
            \ul^2 + 2 \e C_1 \dfrac{q}{\e^2} + \dfrac{q^2}{\e^2}
            \le
            \ul^2 + 2 \sigma B \dfrac{q}{\e^2} + \sigma (\beta + 1) \dfrac{q}{\e^2}
            \le
            \tilde{B}' 
            \left(
                1 + \sigma \dfrac{q}{\e^2}
            \right),            
        \end{align}
        for some positive constant $\tilde{B}'$. 
        Since $\beta$ and $B$ are bounded constants, we can find a positive constant $S_5$ such that 
        \begin{align*}
            |E_5|
            \le
            S_5
            \left(
                1 + \sigma \dfrac{q}{\e^2}
            \right),
        \end{align*}
        holds.

    \end{enumerate}

    From the estiamtes above, we have
    \begin{align*}
        \ml(u^+)
        &\ge
        - (S_2 e^{Lt} + S_3 + S_5)
        +
        \left(
            {S_4} - S_5 \sigma
        \right) \dfrac{q}{\e^2}
        \\
        &=
        - (S_3 + S_5)
        +
        \left(
            {S_4}- S_5\sigma
        \right)
        \sigma \dfrac{\beta e^{- \beta t/\e^2}}{\e^2}
        +
        ({S_4}L \sigma - S_2 - S_5\sigma^2) e^{Lt}
        .
    \end{align*}
    By choosing $\sigma$ small enough and $L$ large enough we finally obtain $\ml(u^+) \ge 0.$

    \item In the set $\{ (x,t) \in \Om \times [0,T - t^\e],~ \tilde{d}_0 \le |d(x,t)| \le 2\tilde{d}_0\}$

    From here, we use $u^+ = (1 - \rho_1(d)) \tilde{u}^+ + \rho_1(d) \rho_2(d) + q$, where $\tilde{u}^+ = \us(z_d;e) + \e \ul(z_d,x,t;e) $. 
    Note that since $|d|$ is bounded below by $\tilde{d}_0$, from the boundedness of $p$ in \eqref{cond:q_bound} we see that 
    \begin{align*}
        \e|z_d|
        \ge
        |d| - \e p
        \ge
        \tilde{d}_0/2. 
    \end{align*}
    Moreover, as we assumed that the function $\rho_1(d) $ is smooth, we can find a constant $C_\rho$ such that
    \begin{align}\label{eq:rho_bound}
        \Vert \rho_1 \Vert_{C^2(\R)}
        \le 
        C_\rho.
    \end{align}
    Also, $\rho_2$ is $\a_+$ if $d \ge \tilde{d}_0$ and $\a_-$ if $d \le - \tilde{d}_0$, we do not need to consider the derivative of $\rho_2$. Moreover, by Lemmas \ref{lem:TW-est}, \ref{lem:TWlin_est}  we obtain that
    \begin{align}
        |\rho_2(d) - \tilde{u}^+|
        &\le
        |\rho_2(d) - \us | + |\e \ul|\nonumber
        \\
        &\le
        C_0 e^{- \la_0|z_d|} + \e C_1 e^{- \la_1|z_d|} \nonumber
        \\
        &\le
        (C_0 + \e C_1) e^{- \tilde\la \tilde{d}_0/2\e},
        \label{eq:rho2-u}
    \end{align}
    where $\tilde{\la} = \min \{ \la_0, \la_1 \} > 0$. 
    With these, we first show the estimates of the derivatives of $u^+$. Straightforward computations give
    \begin{align*}
        u^+_t
        &=
        (1 - \rho_1) \tilde{u}^+_t
        +
        \rho_1' d_t ( \rho_2 -  \tilde{u}^+ )
        + q'
        ,~
        \pvi u^+
        \\
        &=
        (1 - \rho_1) \pvi \tilde{u}^+
        +
        \rho_1' \pvi d ( \rho_2 -  \tilde{u}^+ ).
    \end{align*}
    Next, in view of \eqref{cond:q_bound} and \eqref{eq:r_bounded_1} we obtain that 
    \begin{align}\label{eq:deri_x1_re}
        | \pvi \pvj \tilde u^+ |
        +
        | \pvi \tilde u^+ \pvj \tilde u^+ | 
        \le&
        \dfrac{C_r}{\e^2} e^{- \tilde{\la} |z_d|}
        \le
        \dfrac{C_r}{\e^2} e^{- \tilde{\la} \tilde d_0/2\e}.
    \end{align}
    Similar to this, using \eqref{eq:deri_t}, \eqref{eq:deri_x0} and Lemmas \ref{lem:d}, \ref{lem:TW-est} ,\ref{lem:TWlin_est} we obtain that
    \begin{align}\label{eq:deri_x0_re}
        |\tilde{u}^+_t|
        +
        |\pvi \tilde{u}^+|
        &\le
        \dfrac{C}{\e} e^{- \tilde\la |z_d|} 
        \le 
        \dfrac{C}{\e} e^{- \tilde\la \tilde{d}_0/2\e},
    \end{align}
    for some positive constant $C$.
    To bound these derivatives, we first choose $\e_0 > 0$ small enough that satisfies 
    \begin{align}\label{eq:e^p_bound}
        \dfrac{1}{\e^4} e^{- \tilde{\la} \tilde{d_0}/2\e}
        =
        \dfrac{1}{(\tilde{\la} \tilde{d_0})^4} 
        \left( 
            \dfrac{\tilde{\la} \tilde{d_0} }{\e} 
        \right)^4
        e^{- \tilde{\la} \tilde{d_0}/2\e}
        \le 1,
    \end{align}
    for any $\e \in (0,\e_0)$; this is possible since $z^4 e^{- z/2} \to 0$ as $z \to \infty$.
    Thus, combining \eqref{eq:rho_bound}, \eqref{eq:rho2-u}, \eqref{eq:deri_x1_re}, \eqref{eq:deri_x0_re}, \eqref{eq:e^p_bound} and the fact that $0  < \e < 1, 0 \le \rho_1 \le 1$ we obtain that
    \begin{align*}
        |u_t^+ - q' |
        \le&
        C'\e^3
        \\
        |\pvi u^+|
        \le&
        |\pvi \tilde{u^+}|
        + \Vert \rho_1 \Vert_{C^2(\R)} | (\rho_2 - \tilde{u}^+) \pvi d  | 
        \le
        C' \e^3
        \\
        | \pvi \pvj u^+ |
        \le&
        | \pvi \pvj \tilde{u}^+ |
        +
        | \pvi \pvj (\rho_1 (\rho_2 - \tilde{u}^+)) |
        \\
        \le&
        2| \pvi \pvj \tilde{u}^+ |
        +
        \Vert \rho_1 \Vert_{C^1{(\R)}} 
        \left\{\left| \pvi \tilde{u}^+ \pvj d \right|  +  \left| \pvj \tilde{u}^+ \pvi d \right|\right\}
        \\
        &+
        \Vert \rho_1 \Vert_{C^2{(\R)}} 
        \left| (\rho_2 - \tilde{u}^+) \pvi \pvj d \right|
        \\
        \le&
        C'\e^2,
        \\
        |\pvi u^+ \pvj u^+|
        \le&
        C' \e^2
    \end{align*}
    where $C'$ is some positive constant. With this inequality, noting that $||D||_{C^3(\R)} \le C_D$ by \eqref{cond:D_pos},  we obtain that
    \begin{align}\label{eq:s'1}
        |u^+_t - q' |
        +
        |\pvi(\D(u^+) \pvj u^+)|
        &\le
        S'_1 \e^2 ,
    \end{align}
    where we omitted $\sum_{i,j = 1}^N$ and $S'_1$ is some positive constant.

    We now estimate $f(u^+)$. This time, we make a Taylor expansion at $\rho_2$, which gives
    \begin{align}\label{eq:Taylor'}
        f(u^+)
        =
        f(\rho_2)
        + f'(\rho_2) (\vp' + q)
        + \dfrac{f''(\theta'(x,t))}{2}(\vp' + q)^2,
    \end{align}
    where $\vp' = (1 - \rho_1)(\tilde{u}^+ - \rho_2)$ and $\theta'$ is some constant between $\rho_2$ and $u^+$. Note that, since $\rho_2$ is either $\a_+$ or $\a_-$, we have $f(\rho_2) = 0$. Also, by using  \eqref{eq:rho_bound} and \eqref{eq:e^p_bound} we obtain that
    \begin{align}\label{eq:s'2}
        |f'(\rho_2)\vp'|
        \le
        |f'(\rho_2)(\tilde{u}^+ - \rho_2)|
        \le
        S'_2 \e^4
        \le
        S'_2 \e^2
    \end{align}
    for some positive constant $S'_2$, where first inequality holds since $0 \le \rho_1 \le 1$.  Also, noting that $C_f := - \max\{ f'(\a_+), f'(\a_-) \} > 0$ by \eqref{cond:f-bistable} we obtain
    \begin{align}
        \e^2 q'
        -f'(\rho_2)q 
        &\ge
        - \sigma \beta^2 e^{- \beta t / \e^2}
        + \e^4 \sigma L^2 e^{Lt}
        + C_f q\nonumber
        \\
        &=
        (C_f - \beta) \sigma \beta e^{- \beta t / \e^2}
        + (\e^2 L + C_f) \e^2\sigma L e^{Lt}
        \ge S'_3 q\label{eq:s'3}
    \end{align}
    for some positive constant $S'_3$ and choosing $\beta > 0$ small enough.
    Then, recalling $|\vp'| \le C \e^4$ and $0 \le q \le C \sigma$ we obtain that
    \begin{align}\label{eq:s'4}
        \dfrac{|f''(\theta'(x,t))|}{2}(\vp' + q)^2
        \le S'_4(\e^2 + \sigma q),
    \end{align}
    for some positive constant $S'_3$. With this we can estimate $f(u^+)$.

    With \eqref{eq:s'1}, \eqref{eq:s'2}, \eqref{eq:s'3} and \eqref{eq:s'4}, noting that $q'$ in \eqref{eq:s'1} and \eqref{eq:s'3} cancel with each other we derive that
    \begin{align*}
        \ml(u^+)
        &\ge
        - (S'_1 + S'_2) \e^2 - S'_4
        + (S'_3 - \sigma S'_4)\dfrac{q}{\e^2}
        \\
        &=
        - (S'_1 + S'_2) \e^2 - S'_4
        + (S'_3  - \sigma S'_4) \sigma L e^{Lt}
        \\
        &
        + (S'_3  - \sigma S'_4) \dfrac{\sigma \beta e^{-\beta t/\e^2}}{\e^2}.
    \end{align*}
    Thus, by choosing $\e_0, \sigma > 0$ small enough and $L$ large enough we finally obtain $\ml(u^+) \ge 0$.

    \item In the set $\{ (x,t) \in \Om \times [0,T - t^\e],~ |d(x,t)| \ge 2 \tilde{d}_0\}$

    Since $u^+$ is constant in spatial variable, we only need to prove $q' - f(\rho_2(d) + q)/\e^2 \ge 0$. For Taylor expansion of $f(u^+)$, we can use \eqref{eq:Taylor'}, where $\vp' = 0$ and $\theta'(x,t)$ is some number between $\rho_2(d(x,t))$ and $u^+(x,t)$. With this, and using \eqref{eq:s'3}, \eqref{eq:s'4} gives 
    \begin{align*}
        q' - f(\rho_2(d) + q)/\e^2
        &\ge
        S'_3 \dfrac{q}{\e^2} - S'_4 \left( 1 + \sigma \dfrac{q}{\e^2} \right)
        \\
        &=
        - S'_4
        +
        \sigma (S'_3 - \sigma S'_4) \left( \dfrac{\beta e^{-\beta t/\e^2}}{\e^2} + L e^{Lt} \right).
    \end{align*}
    Thus, by choosing $\sigma$ small enough and $L$ large enough we obtain $\ml(u^+) \ge 0$. This completes the proof of Lemma \ref{lem:Prop_subsuper}.
    \end{enumerate}    
\end{proof}

\subsection{Proof of Theorem \ref{thm:prop}}

We now prove Theorem \ref{thm:prop}. For this, we need two steps: (i) for large enough $K>0$ in $p(t)$ we prove that $u^-(x,t) \leq u^\e(x,t + t^\e) \leq u^+(x,t)$ for $(x,t) \in \Om \times [0,T - t^\e]$ and (ii) we prove the desired result. Once we prove (i) in $\Om \times [0,T - t^\e]$, it is enough to prove the assertion (ii) in $\Om'_T := \{ (x,t) \in \Om \times [0,T - t^\e],~ |d(x,t)| \le \tilde{d}_0\}$; this is because the assertion describes the solution $u^\e$ away from the interface $\Ga_t$ with distance of order $\o(\e)$ and outside of $\Om'_T$ the sub- and super-solutions $u^\pm$ is already close enough to $\a_\pm$.
\begin{enumerate}
    \item[Step 1] We assume that we choose $L, \e_0$ and $\sigma$ such that $u^\pm$ becomes a pair of sub- and super-solutions as in Lemma \ref{lem:Prop_subsuper}. Since $\sigma, \e_0 > 0$ were chosen to be small enough, by choosing smaller if necessary, we can assume that
    \begin{align}\label{eq:thm2_pf0}
        \sigma ( \beta + \e^2_0 L e^{LT}) \le \dfrac{\eta_p}{4}
        ,~~
        \e_0 C_1 
        \le \sigma\beta/4
        \le \eta_p/4,
    \end{align}
    where the last inequality holds since $\sigma \beta \le \eta_p$ by the first inequality.
    Then by letting $\eta_g = \sigma\beta/2$, by choosing $\e_0$ smaller if necessary the result of Theorem \ref{thm:gen} holds for some $M_0 > 0$. By \eqref{cond:gamma0_normal} and \eqref{cond:u0_inout} and the fact that $\Ga_0 = \{ x \in \Om, d(x,0) = 0 \}$  we can find a positive constant $M_1$ such that
    \begin{eqnarray*}
	\text{if } &d(x,0) \leq - M_1 \e~~  &\text{ then } ~~ u_0(x) \le \a - M_0 \e,
	\\
	\text{if } &d(x,0) \geq M_1 \e~~  &\text{ then } ~~ u_0(x) \ge \a + M_0 \e.
    \end{eqnarray*}
    Define step functions $H^\pm(x)$ by
    \begin{align*}
    H^\pm(x)
	:=
	\begin{cases}
		\a_+ \pm {\eta_g}
		&
		~~
		\text{ if }
		d(x,0) \ge \mp M_1 \e,
		\\
		\a_- \pm {\eta_g}
		&
		~~
		\text{ if }
		d(x,0) < \mp M_1 \e.
	\end{cases}
    \end{align*}
    Then the observation above with Theorem \ref{thm:gen} gives that
    \begin{align*}
        H^-(x) 
        \le
        u^\e(x,t^\e)
        \le
        H^+(x)
        ,~~
        \text{for}~~
        x \in \Om.
    \end{align*}

    Next we adjust $u^\pm(x,0)$ to satisfy $u^-(x,0) \le H^-(x) , H^+(x) \le u^+(x,0)$; then by Lemma \ref{lem:comparison} we can bound $u^\e(x, t + t^\e)$ with $u^\pm(x, t)$. We only prove the later inequality; the other inequality can be proved in a similar way. For this, we first take $K>0$ sufficiently large such that 
    \begin{align}\label{eqn:prop_proof-1}
        \us( - M_1 + K; e) &\ge \a_+ - \frac{\eta_g}{2} = \a_+ - \frac{\sigma \beta}{4},
    \end{align}
    for all $e \in \mathbb{S}^{N - 1}$. 
    Then, by \eqref{eq:thm2_pf0} and Lemma \ref{lem:TWlin_est} if $|d(x,0)| \le \tilde{d}_0$ we obtain that 
    \begin{align*}
        u^+(x,0)
        =&
        \us + \e \ul + q(0)
        \\
        \ge&
        \us
        \left(
            \dfrac{d}{\e} + K; \nabla d
        \right)
        - \e C_1 
        + \sigma\beta
        + \e^2 L
        \\
        \ge&
        \us
        \left(
            \dfrac{d}{\e} + K; \nabla d
        \right)
        + 3\sigma\beta/4
    \end{align*}
    Thus, by \eqref{eqn:prop_proof-1} we have
    \begin{align*}
        u^+(x,0)
        \ge
        \a_+ + {\eta_g }
        = H^+(x)
    \end{align*}
     for $- M_1 \e \le d(x,0) \le \tilde{d}_0$. And for $ - \tilde{d}_0 \le d(x,0) \le - M_1 \e $, above computation gives
     \begin{align*}
         u^+(x,0)
         &\ge
         \a_-  + 3 \sigma \b/4
         \ge
         \a_- + \eta_g
         =
         H^+(x)
         .
     \end{align*}
     For $|d(x,0)| \ge \tilde{d}_0$, using \eqref{eq:rho2-u} and \eqref{eq:e^p_bound}, assuming $\e^4(C_0 + \e C_1) \le \sigma \beta/2$ we obtain 
     \begin{align*}
         u^+(x,0)
         &\ge
         \rho_2(d)
         -
         |(1 - \rho_1(d))(\tilde{u}^+ - \rho_2(d))|
         + q(0)
         \\
         &\ge \a_- - \e^4(C_0 + \e C_1) + q(0)
         \ge \a_- - \e^4(C_0 + \e C_1) + \sigma \beta
         \\
         &\ge \a_- + \sigma \beta / 2
         = \a_- + \eta_g
         = H^+(x),
     \end{align*}
     if $d(x,0) \le - \tilde{d}_0$ and
     \begin{align*}     
         u^+(x,0)
         &\ge \a_+ - \e^4(C_0 + \e C_1) + q(0)
         \ge \a_+ - \e^4(C_0 + \e C_1) + \sigma \beta
         \\
         &\ge \a_+ + \sigma \beta / 2
         = \a_- + \eta_g
         = H^+(x),
     \end{align*}
     if $d(x,0) \ge \tilde{d}_0$. This implies that $u^\e(x,t^\e) \le u^+(x,0)$, and similar computations will leads to $u^-(x,0) \le u^\e(x,t^\e)$. Thus, by Lemma \ref{lem:comparison} we proved the assertion; $u^-(x,t) \le u^\e(x,t + t^\e) \le u^+(x,t)$ for $(x,t) \in \Om \times [0,T - t^\e]$.
     
    \item[Step 2]
    We now show the results of Theorem \ref{thm:prop} in $\Om'_T$. Choose $C_p$ large enough such that 
    \begin{align*}
        \us(C_p - L - K; e) \ge \a_+ - \eta_p/2
        ,~~
        \us(-C_p + L + K; e) \le \a_- + \eta_p/2,
    \end{align*}
    for all $e \in \mathbb{S}^{N - 1}$. 
    Thus, since $u^-(x,t) \le u^\e(x,t) \le u^+(x,t)$, if $d(x,t) \ge  C_p\e$ using \eqref{eq:thm2_pf0} we have
    \begin{align*}
        u^\e(x,t + t^\e) 
        &\ge 
        u^-(x,t)
        \\
        &=
        \us(z_d;\nabla d) + \e \ul(z_d;\nabla d) - q
        \\
        &\ge
        \us(C_p - L - K;\nabla d)        
        - \e C_1 - \sigma(\beta + \e^2 L e^{Lt})
        \\
        &\ge
        \a_+ - \eta_p.
    \end{align*}
    And using similar computation, if $d(x,t) \le - C_p\e$ we obtain 
    \begin{align*}
        u^\e(x,t + t^\e)
        \le u^+(x,t)
        \le \a_- + \eta_p.
    \end{align*}
    And lastly, since $|\e \ul| + |q| \le \eta_p/2$ by \eqref{eq:thm2_pf0}, we can see for all $x \in \Om, t \in [0, T - t^\e]$ that 
    \begin{align*}
        \a_- - \eta_p \le u^-(x,t)
        \le
        u^\e(x,t + t^\e)
        \le
        u^+(x,t) \le \a_+ + \eta_p,
    \end{align*}
    which proves the results of Theorem \ref{thm:prop}.

\end{enumerate}

 \section*{Acknowledgement}

The authors thank Yuning Liu for his comment on the approach due to the $\Gamma$-convergence.

\end{document}